\documentclass[11pt]{amsart}
\usepackage{amsfonts}
\usepackage{amsmath}
\usepackage{amssymb}
\usepackage{color}
\usepackage{graphicx}
\usepackage{xspace}
\usepackage{axodraw}
 \font \eightrm=cmr8

 \newcommand{\nc}{\newcommand}

\newtheorem{thm}{Theorem}

\newtheorem{cor}[thm]{Corollary}

\newtheorem{lem}[thm]{Lemma}
\newtheorem{prop}[thm]{Proposition}

\newtheorem{rmk}[thm]{Remark}

\newcommand{\tree}{\scalebox{0.4}{{\parbox{0.5pc}{
  \begin{picture}(30,45) (75,-60)
    \SetWidth{1.5}
    \SetColor{Black}
    \Line(90,-30)(75,-60)
    \Line(90,-30)(105,-60)
    \Line(90,-15)(90,-30)
  \end{picture}}}}}

\newcommand{\treeC}{\scalebox{0.2}{{\parbox{0.5pc}{
 \begin{picture}(90,105) (75,-30)
    \SetWidth{2.4}
    \SetColor{Black}
    \Line(90,0)(75,-30)
    \Line(120,60)(90,0)
    \Line(90,0)(105,-30)
    \Line(105,30)(135,-30)
    \Line(120,60)(165,-30)
    \Line(120,75)(120,60)
  \end{picture}
}}}}

\newcommand{\treeD}{\scalebox{0.2}{{\parbox{0.5pc}{
 \begin{picture}(90,105) (75,-30)
    \SetWidth{2.4}
    \SetColor{Black}
    \Line(90,0)(75,-30)
    \Line(120,60)(90,0)
    \Line(90,0)(105,-30)
    \Line(120,60)(165,-30)
    \Line(120,75)(120,60)
    \Line(150,0)(135,-30)
  \end{picture}
}}}}

\newcommand{\treeF}{\scalebox{0.2}{{\parbox{0.5pc}{
 \begin{picture}(90,105) (75,-30)
    \SetWidth{2.4}
    \SetColor{Black}
    \Line(90,0)(75,-30)
    \Line(120,60)(90,0)
    \Line(120,60)(165,-30)
    \Line(120,75)(120,60)
    \Line(135,30)(105,-30)
    \Line(120,0)(135,-30)
  \end{picture}
}}}}

\newcommand{\treeG}{\scalebox{0.2}{{\parbox{0.5pc}{
 \begin{picture}(90,105) (75,-30)
    \SetWidth{2.4}
    \SetColor{Black}
    \Line(90,0)(75,-30)
    \Line(120,60)(90,0)
    \Line(120,60)(165,-30)
    \Line(120,75)(120,60)
    \Line(105,30)(135,-30)
    \Line(120,0)(105,-30)
  \end{picture}
}}}}

\def\ta1{{\scalebox{0.2}{ 
\begin{picture}(12,12)(38,-38)
\SetWidth{0.5} \SetColor{Black} \Vertex(45,-33){5.66}
\end{picture}}}}

\def\tb2{{\scalebox{0.2}{ 
\begin{picture}(12,42)(38,-38)
\SetWidth{0.5} \SetColor{Black} \Vertex(45,-3){5.66}
\SetWidth{1.0} \Line(45,-3)(45,-33) \SetWidth{0.5}
\Vertex(45,-33){5.66}
\end{picture}}}}

\def\tc3{{\scalebox{0.2}{ 
\begin{picture}(12,72)(38,-38)
\SetWidth{0.5} \SetColor{Black} \Vertex(45,27){5.66}
\SetWidth{1.0} \Line(45,27)(45,-3) \SetWidth{0.5}
\Vertex(45,-33){5.66} \SetWidth{1.0} \Line(45,-3)(45,-33)
\SetWidth{0.5} \Vertex(45,-3){5.66}
\end{picture}}}}

\def\td31{{\scalebox{0.2}{ 
\begin{picture}(42,42)(23,-38)
\SetWidth{0.5} \SetColor{Black} \Vertex(45,-3){5.66}
\Vertex(30,-33){5.66} \Vertex(60,-33){5.66} \SetWidth{1.0}
\Line(45,-3)(30,-33) \Line(60,-33)(45,-3)
\end{picture}}}}

\def\te4{{\scalebox{0.2}{ 
\begin{picture}(12,102)(38,-8)
\SetWidth{0.5} \SetColor{Black} \Vertex(45,57){5.66}
\Vertex(45,-3){5.66} \Vertex(45,27){5.66} \Vertex(45,87){5.66}
\SetWidth{1.0} \Line(45,57)(45,27) \Line(45,-3)(45,27)
\Line(45,57)(45,87)
\end{picture}}}}

\def\tf41{{\scalebox{0.2}{ 
\begin{picture}(42,72)(38,-8)
\SetWidth{0.5} \SetColor{Black} \Vertex(45,27){5.66}
\Vertex(45,-3){5.66} \SetWidth{1.0} \Line(45,27)(45,-3)
\SetWidth{0.5} \Vertex(60,57){5.66} \SetWidth{1.0}
\Line(45,27)(60,57) \SetWidth{0.5} \Vertex(75,27){5.66}
\SetWidth{1.0} \Line(75,27)(60,57)
\end{picture}}}}

\def\tg42{{\scalebox{0.2}{ 
\begin{picture}(42,72)(8,-8)
\SetWidth{0.5} \SetColor{Black} \Vertex(45,27){5.66}
\Vertex(45,-3){5.66} \SetWidth{1.0} \Line(45,27)(45,-3)
\SetWidth{0.5} \Vertex(15,27){5.66} \Vertex(30,57){5.66}
\SetWidth{1.0} \Line(15,27)(30,57) \Line(45,27)(30,57)
\end{picture}}}}

\def\th43{{\scalebox{0.2}{ 
\begin{picture}(42,42)(8,-8)
\SetWidth{0.5} \SetColor{Black} \Vertex(45,-3){5.66}
\Vertex(15,-3){5.66} \Vertex(30,27){5.66} \SetWidth{1.0}
\Line(15,-3)(30,27) \Line(45,-3)(30,27) \Line(30,27)(30,-3)
\SetWidth{0.5} \Vertex(30,-3){5.66}
\end{picture}}}}

\def\thj44{{\scalebox{0.2}{ 
\begin{picture}(42,72)(8,-8)
\SetWidth{0.5} \SetColor{Black} \Vertex(30,57){5.66}
\SetWidth{1.0} \Line(30,57)(30,27) \SetWidth{0.5}
\Vertex(30,27){5.66} \SetWidth{1.0} \Line(45,-3)(30,27)
\SetWidth{0.5} \Vertex(45,-3){5.66} \Vertex(15,-3){5.66}
\SetWidth{1.0} \Line(15,-3)(30,27)
\end{picture}}}}

\def\ti5{{\scalebox{0.5}{ 
\begin{picture}(12,132)(23,-8)
\SetWidth{0.5} \SetColor{Black} \Vertex(30,117){5.66}
\SetWidth{1.0} \Line(30,117)(30,87) \SetWidth{0.5}
\Vertex(30,87){5.66} \Vertex(30,57){5.66} \Vertex(30,27){5.66}
\Vertex(30,-3){5.66} \SetWidth{1.0} \Line(30,-3)(30,27)
\Line(30,27)(30,57) \Line(30,87)(30,57)
\end{picture}}}}

\def\tj51{{\scalebox{0.5}{ 
\begin{picture}(42,102)(53,-38)
\SetWidth{0.5} \SetColor{Black} \Vertex(61,27){4.24}
\SetWidth{1.0} \Line(75,57)(90,27) \Line(60,27)(75,57)
\SetWidth{0.5} \Vertex(90,-3){5.66} \Vertex(60,27){5.66}
\Vertex(75,57){5.66} \Vertex(90,-33){5.66} \SetWidth{1.0}
\Line(90,-33)(90,-3) \Line(90,-3)(90,27) \SetWidth{0.5}
\Vertex(90,27){5.66}
\end{picture}}}}

\def\tk52{{\scalebox{0.5}{ 
\begin{picture}(42,102)(23,-8)
\SetWidth{0.5} \SetColor{Black} \Vertex(60,57){5.66}
\Vertex(45,87){5.66} \SetWidth{1.0} \Line(45,87)(60,57)
\SetWidth{0.5} \Vertex(30,57){5.66} \SetWidth{1.0}
\Line(30,57)(45,87) \SetWidth{0.5} \Vertex(30,-3){5.66}
\SetWidth{1.0} \Line(30,-3)(30,27) \SetWidth{0.5}
\Vertex(30,27){5.66} \SetWidth{1.0} \Line(30,57)(30,27)
\end{picture}}}}

\def\tl53{{\scalebox{0.5}{ 
\begin{picture}(42,102)(8,-8)
\SetWidth{0.5} \SetColor{Black} \Vertex(30,57){5.66}
\Vertex(30,27){5.66} \SetWidth{1.0} \Line(30,57)(30,27)
\SetWidth{0.5} \Vertex(30,87){5.66} \SetWidth{1.0}
\Line(30,27)(45,-3) \SetWidth{0.5} \Vertex(15,-3){5.66}
\SetWidth{1.0} \Line(15,-3)(30,27) \Line(30,57)(30,87)
\SetWidth{0.5} \Vertex(45,-3){5.66}
\end{picture}}}}

\def\tm54{{\scalebox{0.5}{ 
\begin{picture}(42,72)(8,-38)
\SetWidth{0.5} \SetColor{Black} \Vertex(30,-3){5.66}
\SetWidth{1.0} \Line(30,27)(30,-3) \Line(30,-3)(45,-33)
\SetWidth{0.5} \Vertex(15,-33){5.66} \SetWidth{1.0}
\Line(15,-33)(30,-3) \SetWidth{0.5} \Vertex(45,-33){5.66}
\SetWidth{1.0} \Line(30,-33)(30,-3) \SetWidth{0.5}
\Vertex(30,-33){5.66} \Vertex(30,27){5.66}
\end{picture}}}}

\def\tn55{{\scalebox{0.5}{ 
\begin{picture}(42,72)(8,-38)
\SetWidth{0.5} \SetColor{Black} \Vertex(15,-33){5.66}
\Vertex(45,-33){5.66} \Vertex(30,27){5.66} \SetWidth{1.0}
\Line(45,-33)(45,-3) \SetWidth{0.5} \Vertex(45,-3){5.66}
\Vertex(15,-3){5.66} \SetWidth{1.0} \Line(30,27)(45,-3)
\Line(15,-3)(30,27) \Line(15,-3)(15,-33)
\end{picture}}}}

\def\tp56{{\scalebox{0.5}{ 
\begin{picture}(66,111)(0,0)
\SetWidth{0.5} \SetColor{Black} \Vertex(30,66){5.66}
\Vertex(45,36){5.66} \SetWidth{1.0} \Line(30,66)(45,36)
\Line(15,36)(30,66) \SetWidth{0.5} \Vertex(30,6){5.66}
\Vertex(60,6){5.66} \SetWidth{1.0} \Line(60,6)(45,36)
\SetWidth{0.5}
\SetWidth{1.0} \Line(45,36)(30,6) \SetWidth{0.5}
\Vertex(15,36){5.66}
\end{picture}}}}

\nc{\ignore}[1]{{}}
\nc{\mrm}[1]{{\rm #1}}
\nc{\dirlim}{\displaystyle{\lim_{\longrightarrow}}\,}
\nc{\invlim}{\displaystyle{\lim_{\longleftarrow}}\,}
\nc{\vep}{\varepsilon} \nc{\ep}{\epsilon}
\nc{\sigmat}{\widetilde\sigma}
\nc{\ostar}{\overline{*}}

\nc{\mchar}{\mrm{Char}}
\nc{\Hom}{\mrm{Hom}}
\nc{\id}{\mrm{id}}

\nc{\remark}{\noindent{\bf{Remark:}}}
\nc{\remarks}{\noindent{\bf{Remarks:}}}

 \nc{\delete}[1]{}
 \nc{\grad}[1]{^{({#1})}}
 \nc{\fil}[1]{_{#1}}

\nc{\BA}{{\Bbb A}} \nc{\CC}{{\Bbb C}} \nc{\DD}{{\Bbb D}}
\nc{\EE}{{\Bbb E}} \nc{\FF}{{\Bbb F}} \nc{\GG}{{\Bbb G}}
\nc{\HH}{{\Bbb H}} \nc{\LL}{{\Bbb L}} \nc{\NN}{{\Bbb N}}
\nc{\PP}{{\Bbb P}} \nc{\QQ}{{\Bbb Q}} \nc{\RR}{{\Bbb R}}
\nc{\TT}{{\Bbb T}} \nc{\VV}{{\Bbb V}} \nc{\ZZ}{{\Bbb Z}}
\nc{\Cal}[1]{{\mathcal {#1}}}
\nc{\mop}[1]{\mathop{\hbox {\rm #1} }}
\nc{\smop}[1]{\mathop{\hbox {\eightrm #1} }}
\nc{\mopl}[1]{\mathop{\hbox {\rm #1} }\limits}
\nc{\frakg}{{\frak g}}
\nc{\g}[1]{{\frak {#1}}}
\def \restr#1{\mathstrut_{\textstyle |}\raise-8pt\hbox{$\scriptstyle #1$}}
\def \srestr#1{\mathstrut_{\scriptstyle |}\hbox to
  -1.5pt{}\raise-4pt\hbox{$\scriptscriptstyle #1$}}
\nc{\wt}{\widetilde}
\nc{\wh}{\widehat}
\nc{\un}{\hbox{\bf 1}}
\nc{\redtext}[1]{\textcolor{red}{\tt #1}}
\nc{\bluetext}[1]{\textcolor{blue}{#1}}
\nc{\comment}[1]{[[{\tt {#1}}]] }
\nc{\R}{{\mathbb R}}

\nc\fleche[1]{\mathop{\hbox to #1 mm{\rightarrowfill}}\limits}
\def\semi{\mathrel{\times}\kern -.85pt\joinrel\mathrel{\raise 1.4pt\hbox{${\scriptscriptstyle |}$}}}

\begin{document}

\title[Some dendriform equations]
      {Dendriform Equations}

\author{Kurusch Ebrahimi-Fard}
\address{Laboratoire MIA, 
         Universit\'e de Haute Alsace, 
         4 rue des Fr\`eres Lumi\`ere, 
         68093 Mulhouse, France}
 \email{kurusch.ebrahimi-fard@uha.fr}
 \urladdr{http://www.th.physik.uni-bonn.de/th/People/fard/}

\author{Dominique Manchon}
\address{Universit\'e Blaise Pascal,
         C.N.R.S.-UMR 6620
         63177 Aubi\`ere, France}
         \email{manchon@math.univ-bpclermont.fr}
         \urladdr{http://math.univ-bpclermont.fr/~manchon/}

\date{June 10th, 2008\\
\noindent {\footnotesize{${}\phantom{a}$ Mathematics Subject
Classification 2000}: 16W25, 17A30, 17D25, 37C10 Secondary: 05C05, 81T15 }}

\begin{abstract}
We investigate solutions for a particular class of linear equations in dendriform algebras. Motivations as well as several applications are provided. The latter follow naturally from the intimate link between dendriform algebras and Rota--Baxter operators, e.g. the Riemann integral map or Jackson's $q$-integral.       
\end{abstract}

\maketitle

\tableofcontents

\section{Introduction}
\label{sect:intro}

This paper sets out to develop a systematic study of a particular class of linear equations in Loday's dendriform algebras~\cite{Loday}. The guiding principle follows from the intimate connection of such algebras to associative Rota--Baxter algebras~\cite{Aguiar00,Baxter,E,Rota}, where analog equations naturally appear in the context of applications ranging from numerical analysis to renormalization in perturbative quantum field theory. The reader interested in more details may want to consult some of the following references \cite{atkinson,EMP07b,EGP07,EMP07}. 

In the setting of dendriform algebras we obtain the recursive formula for the logarithm of the solutions of the two linear dendriform equations $Y = \un + \lambda a \prec Y$ and $Z = \un - \lambda Z \succ a$, $a \in D$, in $D[[\lambda]]$, where $(D,\prec,\succ)$ is a unital dendriform algebra. We also present the solutions to these equations as an infinite product expansion of exponentials. This way, our work provides a refined approach to the classical Magnus~\cite{Magnus} and Fer expansions~\cite{Fer}, well-known in the field of numerical analysis in the context of approximations of ordinary first order linear differential and integral equations, e.g. see \cite{BCOR98,Iserles84,Iserles}. In our approach both expansions involve the pre-Lie and associative products naturally associated with the underlying dendriform structure and display their interesting interplay. Let us emphasize that the language of dendriform and pre-Lie algebras unveils a hitherto unobserved new structure in these expansions, related to operadic aspects in the context of free pre-Lie algebra \cite{CL}. Moreover, the two linear equations above may be interpreted as first order cases with respect to the parameter $\lambda$. We will introduce equations of arbitrary order in $\lambda$ and show by transforming the higher order equations into dendriform matrix equations that its solutions follow from the first order ones.

Putting the two dendriform products together into a single equation: 
\begin{equation}
\label{gen-order1}
	X = a + \lambda X \succ b + \lambda c \prec X, 
\end{equation}
we find a natural link to Lie bracket flow equations. However, for general dendriform algebras the Lie bracket is replaced by a pre-Lie poduct. We will show how the solutions $Y=Y(c)$ and $Z=Z(b)$ of the first two order recursions above naturally lead to the solution of the general equation \eqref{gen-order1}: 
\begin{equation}
\label{gen-order1sol}
	X = Y*(Y^{-1} \succ a \prec Z^{-1})*Z.
\end{equation}
Here, the product $*$ is the associative product in the dendriform algebra $D$. In fact, we observe that in the special case where $-b=c$, the solution $X$ can be represented in terms of two non-commuting group actions. 

Eventually our approach brings together the works of Magnus \cite{Magnus}, Fer \cite{Fer} and Baxter \cite{Baxter} on linear initial value problems and the corresponding integral equations. A simple example might help to elucidate this point of view. Details will be provided in the sequel. Let $k$ be a field of characteristic zero and $\mathcal F$  an --associative-- $k$-algebra, say, of operator-valued functions on the real line, e.g. smooth $n \times n$ matrix-valued functions. Recall that the solution to the initial value problem:
$$
	\dot{Y}(t)=[Y(t),A(t)]=(YA)(t)-(AY)(t),\quad t>0,\ Y(0)=Y_0,
$$ 
for $A \in \mathcal F$, can be written in terms of the solution of the initial value problem $\dot{U}(t)=A(t)U(t)$, $U(0)=\un$, i.e. $Y(t)=-Ad_{U(t)}(Y_0)=:-U(t)Y_0U^{-1}(t)$. Following Magnus~\cite{Magnus}, see also \cite{Iserles00}, we can write $U(t)=\exp(\Omega(A)(t))$, where $\Omega(A)(0)=0$ and: 
$$
	\dot\Omega(A)(t)=\frac{ad_\Omega}{\exp(ad_\Omega) -1}(A)(t),
$$
which implies:
$$
	Y(t)=-Ad_{\exp(\Omega(A)(t))}(Y_0).
$$ 
Here, we may remark that Magnus' exponential solution gives rise to the following nontrivial identity:
$$
	{\mathcal{T}}\!\!\exp(I(A)(t)) := \un + I(A)(t) + I(AI(A))(t) + I(AI(AI(A))) + \cdots =  \exp(\Omega(A)(t)),
$$
$I(X)(t)=\int_0^t X(s)ds$ denotes the Riemann integral on $\mathcal F$. The expression on the left hand side of the first equality is known as the time-ordered exponential. i.e. the ${\mathcal{T}}$-operation denotes the time ordering operator \cite{OteoRos}. We may rewrite the Lie bracket flow equation for $A \in \mathcal F$, as follows:
\begin{equation}
\label{eq:IVP}
	\dot{Y}(t) = [I(\dot{Y})(t)+Y_0,A(t)] = [Y_0,A(t)] + [I(\dot{Y})(t),A(t)] = A_0(t) + [I(\dot{Y})(t),A(t)].
\end{equation}
The crucial observation is that recursion (\ref{eq:IVP}) actually takes place inside a pre-Lie algebra. Indeed, the integration by parts rule implies that the product defined in terms of the Riemann integral and the ordinary Lie bracket, i.e. $X \rhd Z := [I(X),Z]$ satisfies the --left-- pre-Lie identity, leading to the pre-Lie flow equation:
\begin{equation}
\label{eq:pre-LieIVP}
	\dot{Y}(t) = A_0(t) + (\dot{Y} \rhd A)(t),\qquad \dot{Y}(0)=A_0(0)=[Y_0,A(0)]. 
\end{equation}	
One may interpret this as a linear initial value problem on a pre-Lie algebra. When expanding the Lie bracket, $[I(X),Z]=I(X)Z - ZI(X)$, the newly introduced pre-Lie product, $X \rhd Z=[I(X),Z]$ on $\mathcal F$, can be written as a linear combination of two binary compositions, $X\succ Z:=I(X)Z$, and $Z\prec X:=ZI(X)$, which we call left- and right-dendriform products, respectively: 
 \allowdisplaybreaks{
\begin{eqnarray}
\label{eq:dendriformIVP1}
	\dot{Y}(t) = A_0(t) + (\dot{Y} \succ A)(t) - ( A \prec \dot{Y})(t),\qquad \dot{Y}(0)=A_0(0). 
\end{eqnarray}}
The rules these left- and right-dendriform compositions have to satisfy so as to provide a pre-Lie product are called dendriform axioms \cite{Loday}. In the next section we will present them in detail. Iserles in \cite{Iserles01} further generalized the above Lie bracket flow by looking at the initial value problem $\dot{Y}=YA-BY$, $t>0$, $Y(0)=Y_0$ which can be written: 
 \allowdisplaybreaks{
\begin{eqnarray}
	\dot{Y}(t)=YA(t)-BY(t)&=&(I(\dot{Y})(t)+Y_0)A(t)-B(t)(I(\dot{Y})(t)+Y_0) \nonumber\\
				         &=& C(t) + (I(\dot{Y})A)(t) - (BI(\dot{Y}))(t)\label{eq:iserlesIVP},
\end{eqnarray}	
with $C(t):=Y_0A(t)-B(t)Y_0$. Using the dendriform products, $\prec$ and $\succ$, this writes elegantly as follows:
\begin{equation}
\label{eq:dendriformIVP2}
	\dot{Y} = C + \dot{Y} \succ A - B \prec \dot{Y} , \qquad \dot{Y}(0)=C(0)=Y_0A(0)-B(0)Y_0. 	
\end{equation}	
As we will see in the next section, the surprising fact is, that the same dendriform rules also give rise to an associative product, and both together, the associative and the pre-Lie product may be used to find an exponential solution for Iserles' problem (\ref{eq:dendriformIVP2}), and hence also for the above recursions (\ref{eq:pre-LieIVP}), respectively (\ref{eq:dendriformIVP1}), though in the more general context of dendriform algebras. 

The link to Baxter's work \cite{Baxter} follows naturally from the fact that we solve the above recursions using purely dendriform algebra, which allows us to replace the Riemann integral map $I$ by more general integral type operators, e.g. Rota--Baxter operators of arbitrary weight \cite{E}, including Jackson's $q$-integral as well as Riemann sums.

Finally, we will present several applications at the end of the article, including the noncommutative generalization of a Spitzer type identity first presented in a commutative context by Vogel in \cite{Vogel}, coming from probability theory. A link to Riccati's differential equation is outlined, which we hope to deepen in the near future. We will use planar rooted trees to encode the pre-Lie Magnus expansion and also comment on an observation made in an earlier publication~\cite{EM}, indicating a reduced number of terms in the pre-Lie Magnus and Fer expansions due to the pre-Lie relation.\\ 

The paper is organized as follows. In section \ref{sect:dpse} we recall the definition of a dendriform algebra and introduce the associated left and right pre-Lie products. The augmentation $\overline D$ of a dendriform algebra $D$ by a unit $\un$ is also recalled (\cite{C}, \cite{R}, \cite{R2}). In section \ref{sect:main} we introduce our most general equation in $\overline D[[\lambda]]$: it is said to be of degree $(m,n)$ as it involves monomials of degree $\le m+1$ (resp. $n+1$) with respect to the right (resp. left) dendriform product $\succ$ (resp. $\prec$). We prove in detail that \eqref{gen-order1sol} gives a complete solution of the case $(m,n)=(1,1)$, which corresponds to the equation \eqref{gen-order1} for $a,b,c \in D$, in terms of the solutions of the cases of degree $(1,0)$ and $(0,1)$. The particular (and simpler) case, where $a=\un$ and $b,c\in D$ is also considered. Sections \ref{sect:preLieMag} and \ref{sect:preLieFer} are devoted to the pre-Lie Magnus expansion and the pre-Lie Fer expansion respectively, recalling the results of \cite{EM}. In section \ref{sect:secondorder} we introduce matrix dendriform algebras, which allow us to solve the equations of degree $(2,0)$ and $(0,2)$. The procedure is generalized to equations of degree $(m,0)$ and $(0,n)$ in the following section. Finally, section \ref{sect:appl} is devoted to applications in the context of Rota--Baxter algebras, which furnish a large collection of dendriform algebras \cite{E}. We touch initial value problems in matrix algebras \cite{Iserles01}, an identity discovered by W.~Vogel \cite{Vogel}, and show how the Riccati differential equation is linked with the degree $(2,0)$ dendriform equation. Eventually, we introduce a noncommutative Butcher series like formula for the pre-Lie Magnus expansion and indicate how the pre-Lie structure may be used to reduce the number of terms in it.


\section{Dendriform power sums expansions}
\label{sect:dpse}

Loday's dendriform algebra may be seen at the same time as an associative as well as a pre-Lie algebra. Indeed, it was found that in any dendriform algebra both products can be written as particular linear combinations of the dendriform left and right binary compositions. Main examples of dendriform algebras are provided by the shuffle and quasi-shuffle algebra as well as associative Rota--Baxter algebras. The link between dendriform algebras and the latter is quite natural as shuffle type products and their noncommutative generalizations appear in the construction of free associative Rota--Baxter algebras~\cite{eg2005}. 

We briefly introduce the setting of dendriform algebra. Let $k$ be a field of characteristic zero. Recall that a {\sl dendriform algebra\/}~\cite{Loday} over the field $k$ is a $k$-vector space $A$ endowed with two bilinear operations, denoted $\prec$ and $\succ$ and called right and left products, respectively, subject to the three axioms below:
 \allowdisplaybreaks{
\begin{eqnarray}
  (a\prec b)\prec c  &=& a\prec(b \prec c + b \succ c)        \label{A1}\\
  (a\succ b)\prec c  &=& a\succ(b\prec c) 				 \label{A2}\\
   a\succ(b\succ c)  &=& (a \prec b + a \succ b)\succ c        \label{A3}.
\end{eqnarray}}
One readily verifies that these relations yield associativity for the product 
\begin{equation}
\label{dendassoc}
	a * b := a \prec b + a \succ b.  
\end{equation}	
However, at the same time the dendriform relations imply that the bilinear products
$\rhd$ and $\lhd$ defined by:
\begin{equation}
\label{def:prelie}
    a \rhd b:= a\succ b-b\prec a,
    \hskip 12mm
    a \lhd b:= a\prec b-b\succ a,
\end{equation}
are {\sl{left pre-Lie}} and {\sl{right pre-Lie}}, respectively, which means that we have:
 \allowdisplaybreaks{
\begin{eqnarray}
    (a\rhd b)\rhd c-a\rhd(b\rhd c)&=& (b\rhd a)\rhd c-b\rhd(a\rhd c),\label{prelie1}\\
    (a\lhd b)\lhd c-a\lhd(b\lhd c)  &=& (a\lhd c)\lhd b-a\lhd(c\lhd b).\label{prelie2}
\end{eqnarray}}
Observe that $a \rhd b = - b \lhd a$. The associative operation $*$ and the pre-Lie operations $\rhd$, $\lhd$ all define the same Lie bracket:
\begin{equation}
    [\![a,b]\!]:=a*b-b*a=a\rhd b-b\rhd a=a\lhd b-b\lhd a.
\end{equation}
Observe that:
$$
	a*b + b \rhd a = a \succ b + b \succ a.
$$
We stress here that in the commutative case (commutative dendriform algebras are also named {\sl Zinbiel algebras\/} \cite{Loday2}, \cite{Loday}), the left and right operations are further required to identify, so that $a \succ b
= b \prec a$. In this case both pre-Lie products vanish. A natural example of a commutative dendriform algebra is given by the shuffle algebra in terms of half-shuffles \cite{Sch58}. Any associative algebra $A$ equipped with a linear integral like map $I: A \to A$ satisfying the integration by parts rule also gives a dendriform algebra, when $a \prec b:=aI(b)$ and $a \succ b:=I(a)b$. 

Loday and Ronco introduced in~\cite{LodayRonco} the notion of {\textsl{tridendriform algebra}} $T$ equipped with three binary operations, $<, >$ and $\bullet$, satisfying seven dendriform type axioms:
 \allowdisplaybreaks{
\begin{eqnarray}
    && (x < y) < z = x < (y \star z),   \hskip 4mm
       (x > y) < z = x > (y < z),       \hskip 4mm
       (x \star y) > z = x > (y > z),   \\
    (x > y) \bullet z \!\!\!&=&\!\!\! x > (y \bullet z),   \hskip 1mm
       (x < y) \bullet z = x \bullet (y> z),    \hskip 1mm
       (x \bullet y) < z = x \bullet (y < z),   \hskip 1mm
       (x \bullet y) \bullet z = x \bullet (y \bullet z), \notag
       \label{DT}
\end{eqnarray}}
yielding associativity for the product $x \star y := x < y + x > y + x \bullet y$. First, observe that the category of dendriform algebras can be identified with the subcategory of objects in the category of tridendriform algebras with $\bullet=0$. Moreover, one readily verifies for a tridendriform algebra $(T,<,>,\bullet)$ that $(D_T,\prec_\bullet,\succ)$, where $\prec_\bullet :=\ < +\ \bullet$ and $\succ:=>$, is a dendriform algebra~\cite {E}. Any associative algebra $A$ closed under a finite Riemann sum operator $S: A \to A$ satisfying a corresponding modified integration by parts rule taking the diagonal into account gives a tridendriform algebra, when $a \prec b:=aS(b)$, $a \succ b:=S(a)b$ and $a \bullet b:= \mp ab$. Here, the sign of the last composition depends on whether the diagonal terms must be subtracted or added, respectively.

Let $\overline A = A \oplus k.\un$ be our dendriform algebra augmented by a unit $\un$, such that:
\begin{equation}
\label{unit-dend}
    a \prec \un := a =: \un \succ a
    \hskip 12mm
    \un \prec a := 0 =: a \succ \un,
\end{equation}
implying $a*\un=\un*a=a$. Note that $\un*\un=\un$, but that $\un \prec \un$ and $\un \succ \un$ are not defined, see~\cite{R}, \cite{C} for more details.

We recursively define the following set of so-called left- respectively right-dendriform iterated elements of degree $n \in \mathbb{N}$ in $\overline{A}[[\lambda]]$ for fixed $x_1,\ldots, x_n \in A$:
 \allowdisplaybreaks{
\begin{eqnarray*}
	w^{(n)}_{\succ}(x_1,\ldots, x_n) &:=&   \bigl(\cdots (x_1 \succ x_2) \succ x_3  \cdots \bigr) \succ x_{n} \\
	w^{(n)}_{\prec}(x_1,\ldots, x_n) &:=&  x_1 \prec \bigl( x_2 \prec \cdots (x_{n-1} \prec x_n)  \cdots \bigr).
\end{eqnarray*}}
Defining $w^{(0)}_{\prec}(x_1,\ldots, x_n) := \un =:w^{(0)}_{\succ}(x_1,\ldots, x_n)$, we may write these  left- respectively right-dendriform iterated elements more compactly: 
 \allowdisplaybreaks{
\begin{eqnarray*}
    w^{(n)}_{\succ}(x_1,\ldots, x_n) &:=& \bigl(w^{(n-1)}_\succ(x_1,\ldots, x_{n-1})\bigr)\succ x_n\\
    w^{(n)}_{\prec}(x_1,\ldots, x_n) &:=& x_1 \prec \bigl(w^{(n-1)}_\prec(x_2,\ldots, x_n)\bigr).
\end{eqnarray*}}
In case that $x_1=\cdots = x_n=x$ we simply write $w^{(n)}_{\prec}(x,\ldots, x)= w^{(n)}_{\prec}(\{x\}_n)=w^{(n)}_{\prec}(x)$ and 
 $w^{(n)}_{\succ}(x,\ldots, x)= w^{(n)}_{\succ}(\{x\}_n)=w^{(n)}_{\succ}(x)$.

Let us recall from Chapoton and Ronco~\cite{C, R, R2} that, on the free dendriform algebra on one generator $a$, augmented by a unit element, there is a Hopf algebra structure with respect to the associative product~(\ref{dendassoc}). The elements $w^{(n)}_{\succ} :=w^{(n)}_{\succ}(a)$ generate a cocommutative graded connected Hopf subalgebra $(H,*)$ with
coproduct:
 \allowdisplaybreaks{
\begin{eqnarray}
    \Delta(w^{(n)}_{\succ})=w^{(n)}_{\succ} \otimes \un + \un \otimes w^{(n)}_{\succ}
    + \sum_{0<m<n} w^{(m)}_{\succ} \otimes w^{(n-m)}_{\succ},
\label{Hopf}
\end{eqnarray}}
and antipode $S(w^{(n)}_{\succ})=(-1)^{n} w^{(n)}_{\prec}$. It is easy to check that the 
$w^{(n)}_{\succ}$ generate a free associative subalgebra of the free dendriform algebra on $a$ for the
associative product, so that one can use the previous formula for the coproduct action on $w^{(n)}_{\succ}$ as
a definition of the Hopf algebra structure on $H$. As an important consequence, it follows that $H$ is isomorphic, 
as a Hopf algebra, to the Hopf algebra of noncommutative symmetric functions \cite{Gelfand}.  

We also define the following set of iterated left and right pre-Lie products~(\ref{def:prelie}). For $n>0$, let $x_1,\ldots,x_n \in A$:
 \allowdisplaybreaks{
\begin{eqnarray}
    \ell^{(n)}(x_1,\dots,x_n) &:=&
    \bigl( \cdots  (x_1 \rhd x_2) \rhd x_3  \cdots \bigr) \rhd x_{n} 
\label{leftRBpreLie}\\
    r^{(n)}(x_1,\dots,x_n) &:=&
    x_1 \lhd \bigl( x_2 \lhd \cdots (x_{n-1} \lhd x_n) \cdots \bigr).
\label{rightRBpreLie}
\end{eqnarray}}
More compactly, for a fixed single element $x \in A$ we may write for $n>0$:
 \allowdisplaybreaks{
\begin{eqnarray}
    \ell^{(n+1)}(x) = \bigl(\ell^{(n)}(x)\bigr)\rhd x
\quad\ {\rm{ and }} \quad\
    r^{(n+1)}(x) = x \lhd \bigl(r^{(n)}(x)\bigr)
\label{def:pre-LieWords}
\end{eqnarray}}
where $\ell^{(1)}(x):=x=:r^{(1)}(x)$.

Let us define the exponential and logarithm map in terms of the associative
product~(\ref{dendassoc}). For any $x \in \lambda \overline A[[\lambda]]$:
$$
	\exp^*(x):=\sum_{n \geq 0} x^{*n}/n!  
	\quad\ {\rm{resp.}} \quad\ 
	\log^*(\un+x):=-\sum_{n>0}(-1)^nx^{*n}/n. 
$$


\section{Power sums expansions in unital dendriform algebras}
\label{sect:main}

Let $(A,\prec,\succ)$ be a dendriform algebra, $\overline A$ its augmentation by a unit $\un$. Let us introduce the following {\it{equation of degree $(m,n)$}} in $\overline A[[\lambda ]]$:
\begin{equation}
\label{maxEq}
     X  =  a +  \sum_{q=1}^{m} \lambda^{q} \omega^{(q+1)}_\succ(X,b_{q1},\ldots, b_{qq})  
                +  \sum_{p=1}^{n} \lambda^{p} \omega^{(p+1)}_\prec(c_{p1},\ldots,c_{pp},X),
\end{equation}
with $m, n \in \mathbb{N}_+$ and $a$, $b_{ij}$, $1\le j \le i \le m$, $c_{ij}$, $1\le j \le i \le n$ in $A$. Here, the degree $deg(X) \in \mathbb{N}^2_+$ of $Z$ is defined as the maximal degree of the left- respectively right-dendriform iterated elements in $X$ minus one. In this work we give a detailed account on the solution of the cases $(m,0)$, $(0,n)$ and $(1,1)$, leaving the general case for further studies. 

We will first explore the {\it{equation of degree $(1,1)$}}:
\begin{equation}
\label{eq:degree1relation}
     X = a +   \lambda \omega^{(2)}_\succ(X,b)  
              +   \lambda \omega^{(2)}_\prec(c,X).
\end{equation}
in $\overline A[[\lambda ]]$ for fixed elements $a,b,c \in A$. 


\subsection{Equation of degree $(1,1)$}
\label{ssect:degree2case}

We first study the case where $deg(Z)=(1,1)$ in (\ref{maxEq}). Its solution follows from the solutions of the cases $(1,0)$ and $(0,1)$. Later, we will see how the more general cases $(m,0)$ and $(0,n)$, $m$ and $n$ bigger than one reduce to the cases $(1,0)$ and $(0,1)$ by embedding the dendriform algebra $(A,\prec,\succ)$ into the matrix dendriform algebra $\Cal M_N(A)$ of size $N=N_{\prec}$ and $N=N_{\succ}$, respectively, with 
$$
	N_{\succ}:=1 + \frac{(m-1)m}{2} \quad\ {\rm{and}} 
	\quad\
	\ N_{\prec}:=1 + \frac{(n-1)n}{2}.
$$ 

Relation~(\ref{maxEq}) for $m=n=1$ simplifies to the following equation:
\begin{equation}
\label{eq:degree1relation}
     X = a +   \lambda X \succ b +  \lambda c \prec X,
\end{equation}
in $\overline A[[\lambda ]]$ for fixed elements $a,b,c \in A$.  

First, observe that for $c=-b \in A$, relation (\ref{eq:degree1relation}) reduces to a degree one recursion involving the dendriform left pre-Lie product~(\ref{def:prelie}):
\begin{equation}
\label{eq:pre-Lierelation}
     X = a +   \lambda X \rhd b .
\end{equation}
in $\overline A[[\lambda ]]$ for fixed elements $a,b \in A$. For this pre-Lie type recursion the special case $a=\un$ leads to the simple solution $X=\un$, since $\un \rhd x = x \lhd \un = 0$. 
  
\smallskip
  
Let us start to find solutions of (\ref{eq:degree1relation}). 

\begin{prop}\label{prop:a=1} 
Let  $(\overline{A},\prec,\succ)$ be a unital dendriform algebra. Let $Y=Y(c)$ and $Z=Z(b)$, $c,b \in A$ be solutions of the  recursions:
\begin{equation}
\label{eq:relationYZ}
	Y = \un + \lambda c  \prec Y  
	\quad\ {\rm{resp.}} \quad\ 
	 Z = \un + \lambda Z \succ b,
\end{equation}
of degree $deg(Y)=(0,1)$ respectively $deg(Z)=(1,0)$, in $\overline A[[\lambda]]$. One verifies that the product $ Y*Z $ gives a solution to equation (\ref{eq:degree1relation}) for the particular case $a=\un \in \overline{A}$. 
\end{prop}
\begin{proof}
Indeed, a simple calculation shows that:
 \allowdisplaybreaks{
\begin{eqnarray*}
	 Y*Z &=& \bigl(\un +  \lambda c  \prec Y\bigr)* \bigl(\un + Z \succ \lambda b\bigr) \\
	       &=& \un +  \lambda c  \prec Y + Z \succ  \lambda b +  \bigl( \lambda c  \prec Y\bigr)* \bigl(Z \succ  \lambda b\bigr)\\
	       &=& \un +  \lambda c  \prec Y + Z \succ  \lambda b +  \lambda^2  \bigl(c  \prec Y\bigr) \prec \bigl(Z \succ b\bigr) 
	                                      			  			          + \lambda^2  \bigl(c  \prec Y\bigr) \succ \bigl(Z \succ b\bigr)\\
	       &=& \un +  \bigl( \lambda c  \prec Y\bigr) \prec \bigl( \un + Z \succ  \lambda b \bigr) 
		              +  \bigl(\un +  \lambda c  \prec Y\bigr) \succ  \bigl(Z \succ  \lambda b\bigr)\\
		&=& \un +  \bigl( \lambda c  \prec Y\bigr) \prec Z + Y  \succ  \bigl(Z \succ  \lambda b\bigr)\\
		&=& \un +   \lambda c  \prec \bigl(Y * Z\bigr) + \bigl(Y * Z\bigr) \succ  \lambda b.
\end{eqnarray*}}
\end{proof}
It is obvious that formal solutions to (\ref{eq:relationYZ}) are given by:
\begin{equation}
\label{devt}
    Y = \sum_{n \geq 0} \lambda^nw^{(n)}_{\prec}(c)
    \hskip 15mm {\rm{resp.}} \hskip 15mm
    Z = \sum_{n \geq 0} \lambda^nw^{(n)}_{\succ}(b).
\end{equation}

Notice that, due to the definition of the Hopf algebra structure on $H$, these two series behave as group-like elements with respect to the coproduct $\Delta$, see  (\ref{Hopf}), (up to the extension of the scalars from $k$ to $k[\lambda]$ and the natural extension of the Hopf algebra structure on $H=\bigoplus_{n\ge 0}H_n$ to its completion $\hat H=\prod_{n\ge 0}H_n$ with respect to the grading). Hence, we remark here that with respect to equations (\ref{eq:relationYZ}) the solutions $\tilde{Y}=\tilde{Y}(c)$  and $\tilde{Z}=\tilde{Z}(b)$ of the equations 
\begin{equation}
\label{eq:relationYZinverse}
	\tilde{Y} =\un - \lambda \tilde{Y} \succ c 
	    \hskip 15mm {\rm{and}} \hskip 15mm
	\tilde{Z} = \un - \lambda b \prec \tilde{Z} 
\end{equation}
satisfy:
$$
	\tilde{Y}*Y= \un  = Y * \tilde{Y} 
		\quad\ {\rm{resp.}} \quad\ 
	\tilde{Z}*Z  = \un = Z * \tilde{Z}.	
$$
Indeed we have, using \eqref{devt} and the antipode:
\begin{eqnarray*}
	\tilde Y&=&\sum_{n\ge 0}(-1)^n\lambda^nw_\succ^{(n)}(c)\\
		   &=&S\Big(\sum_{n\ge 0}\lambda^nw_\prec^{(n)}(c)\Big)=S(Y)=Y^{-1},
\end{eqnarray*}
and similarly for $Z$. Hence, we may write 
$\tilde{Y}=Y^{-1}$, and $\tilde{Z}=Z^{-1}$, respectively. 								

We now slightly generalize equations (\ref{eq:relationYZ}) by allowing them to start with a non-unital element:
\begin{equation}
\label{eq:relationYZa}
	E =   a + \lambda b  \prec E  
	\quad\ {\rm{resp.}} \quad\ 
	F =   c + \lambda F \succ d,  
\end{equation}
in $\overline A[[\lambda ]]$ for fixed elements $a, b, c, d \in A$.  

Recall the recursively defined set of  left- respectively right-dendriform iterated elements of degree $n \in \mathbb{N}$ in $\overline{A}[[\lambda]]$, $n>0$:
 \allowdisplaybreaks{
\begin{eqnarray*}
    w^{(n+1)}_{\prec}(\{y\}_n,x) &:=& y \prec \bigl(w^{(n)}_\prec(\{y\}_{n-1},x)\bigr),\\
    w^{(n+1)}_{\succ}(x,\{y\}_n) &:=& \bigl(w^{(n-1)}_\succ(x,\{y\}_{n-1})\bigr)\succ y,
\end{eqnarray*}}
in $A$ and for fixed $x,y \in A$. In $\overline A[[\lambda ]]$ we find immediately the formal solutions for the equations (\ref{eq:relationYZa}):
\begin{equation*}
    E =  a + \sum_{n \geq 2} \lambda^{n-1} w^{(n)}_{\prec}(\{b\}_{n-1},a)
    \hskip 15mm {\rm{and}} \hskip 15mm
    F =  c + \sum_{n \geq 2} \lambda^{n-1} w^{(n)}_{\succ}(c,\{d\}_{n-1}).
\end{equation*}

Our first result is the following:

\begin{lem}\label{lem:m=1eq} Let  $\overline{A}$ be a unital dendriform algebra. Let $Y=Y(b)$ and $Z=Z(d)$, $b,d \in A$, be solutions in $\overline A[[\lambda ]]$ to the equations (\ref{eq:relationYZ}), respectively. For $a,c \in A$ the equations (\ref{eq:relationYZa}) are solved by: 
\begin{equation}
\label{eq:relationYZaSol}
	E =  Y * \bigl( Y^{-1} \succ  a\bigr) 
	\hskip 15mm {\rm{and}} \hskip 15mm
       F = \bigl(   c \prec Z^{-1} \bigr) * Z 
\end{equation}
in $\overline A[[\lambda ]]$, respectively.
\end{lem}
\begin{proof}
Let us verify this for the first equation. The other case works analogously.
 \allowdisplaybreaks{
\begin{eqnarray*}
	Y*\bigl( Y^{-1}  \succ  a\bigr) &=& \bigl( \un + \lambda b  \prec Y \bigr) * \bigl(Y^{-1} \succ a \bigr)\\
	                                               &=& Y^{-1} \succ a + \bigl( \lambda b  \prec Y \bigr) * \bigl(Y^{-1} \succ a \bigr)\\
	                                               &=& Y^{-1} \succ a + \bigl( \lambda b  \prec Y \bigr) \succ \bigl(Y^{-1} \succ a \bigr)
	                                                                               + \bigl( \lambda b  \prec Y \bigr) \prec \bigl(Y^{-1} \succ a \bigr)\\
	                                               &=& Y^{-1} \succ a + \Bigl(\bigl( \lambda b  \prec Y \bigr) *  Y^{-1} \Bigr) \succ  a
	                                                                               +  \lambda b \prec  \Bigl( Y * \bigl(Y^{-1} \succ a \bigr) \Bigr) \\ 
	                                               &=& Y^{-1} \succ a + \Bigl(\bigl( Y - \un \bigr) *  Y^{-1} \Bigr) \succ  a
	                                                                               +  \lambda b \prec  \Bigl( Y * \bigl(Y^{-1} \succ   a \bigr) \Bigr) \\
	                                               &=& a +  \lambda b \prec  \Bigl( Y * \bigl(Y^{-1} \succ  a \bigr) \Bigr).                                                                                                                                                     
\end{eqnarray*}}
\end{proof}

An important remark concerning the restriction $a,c \in A$ is in order. The equations (\ref{eq:relationYZa}) reduce to the ones in (\ref{eq:relationYZ}) for $a = c = \un$. However, for the solutions in (\ref{eq:relationYZaSol}) we must be careful with respect to the particular properties of the dendriform unit, i.e. we have to avoid $\un \prec \un$ respectively $\un \succ \un$,  see (\ref{unit-dend}). Indeed, for $a \in A$ one may rewrite (\ref{eq:relationYZaSol}):
$$
	E =  Y * \bigl( Y^{-1} \succ a\bigr) = Y * \bigl( \un \succ  a  - \lambda (Y^{-1} \succ b) \succ a   \bigr) 
	       						      =   Y * a - \lambda Y* (Y^{-1} \succ b) \succ a. 
$$
Written in this form we may now put $a = \un$, keeping in mind that $b \neq \un$. Hence, we find $E = Y$, i.e. that $E$, with $b$ instead of $c$, is a solution of the first relation in (\ref{eq:relationYZ}) as expected. From this we conclude our next result, i.e. the solution to equation (\ref{eq:degree1relation}). 

\begin{thm}\label{thm:dendsolution1} Let  $\overline{A}$ be a dendriform algebra augmented by a
unit $\un$~(\ref{unit-dend}). Let $Y=Y(c)$ and $Z=Z(b)$, $c,b \in A$, be solutions in $\overline A[[\lambda ]]$ to the equations $Y = \un + \lambda c  \prec Y$,  $Z = \un + \lambda Z \succ b$, of degree $(0,1),\ (1,0)$ respectively. Then the equation:
\begin{equation*}
     X = a +   \lambda X \succ b +  \lambda c \prec X,
\end{equation*}
of degree $deg(X)=(1,1)$ and $a \in A$ is solved by: 
\begin{equation}
\label{eq:relation1Sol}
	X = Y * \bigl( Y^{-1} \succ a  \prec Z^{-1} \bigr) * Z.
\end{equation}
\end{thm}

\begin{proof}
Let us verify this in detail. Recall equations  (\ref{eq:relationYZinverse}):
 \allowdisplaybreaks{
\begin{eqnarray*}
 \lefteqn{Y * \bigl( Y^{-1} \succ a  \prec Z^{-1} \bigr) * Z 
 = \Bigl(\bigl( \un + \lambda c  \prec Y \bigr)* \bigl( Y^{-1} \succ a  \prec Z^{-1} \bigr)\Bigr) * Z} \nonumber \\
 &=& \bigl( Y^{-1} \succ a  \prec Z^{-1} \bigr) * Z 
 			+ \Bigl( \bigl( \lambda c  \prec Y \bigr) \prec \bigl( Y^{-1} \succ a  \prec Z^{-1} \bigr)\Bigr) * Z \\
        &&\hspace{5cm}
			+ \Bigl(\bigl( \lambda c  \prec Y \bigr) \succ \bigl( Y^{-1} \succ a  \prec Z^{-1} \bigr)\Bigr) * Z \nonumber\\
 &=& \bigl( Y^{-1} \succ a  \prec Z^{-1} \bigr) * Z 
 			+ \Bigl( \lambda c  \prec \bigl( Y * \bigl( Y^{-1} \succ a  \prec Z^{-1} \bigr)\bigr)\Bigr) * Z \\
        &&\hspace{5cm}
			+ \Bigl(\bigl(\bigl( \lambda c  \prec Y \bigr) * Y^{-1} \bigr)\succ \bigl(a  \prec Z^{-1} \bigr)\Bigr) * Z \nonumber\\
 &=& \bigl( Y^{-1} \succ a  \prec Z^{-1} \bigr) * Z 
 		+ \Bigl( \lambda c  \prec \bigl( Y * \bigl( Y^{-1} \succ a  \prec Z^{-1} \bigr)\bigr)\Bigr) \succ Z \\
        &&
  			+ \Bigl( \lambda c  \prec \bigl( Y * \bigl( Y^{-1} \succ a  \prec Z^{-1} \bigr)\bigr)\Bigr) \prec Z 		
			   + \Bigl(\bigl(\bigl( \lambda c  \prec Y \bigr) * Y^{-1} \bigr)\succ \bigl(a  \prec Z^{-1} \bigr)\Bigr) * Z \nonumber\\
 &=& \bigl( Y^{-1} \succ a  \prec Z^{-1} \bigr) * Z 
 		+ \Bigl( \lambda c  \prec \bigl( Y * \bigl( Y^{-1} \succ a  \prec Z^{-1} \bigr)\bigr)\Bigr) \succ Z \\
        &&
   		+ \lambda c  \prec \Bigl( Y * \bigl( Y^{-1} \succ a  \prec Z^{-1} \bigr)*Z\Bigr)			
		  + \Bigl(\bigl(\bigl( Y - \un \bigr) * Y^{-1} \bigr)\succ \bigl(a  \prec Z^{-1} \bigr)\Bigr) * Z\nonumber\\
 &=& \Bigl( \lambda c  \prec \bigl( Y * \bigl( Y^{-1} \succ a  \prec Z^{-1} \bigr)\bigr)\Bigr) \succ Z
   	+ \lambda c  \prec \Bigl( Y * \bigl( Y^{-1} \succ a  \prec Z^{-1} \bigr)*Z\Bigr) 
	 + \bigl(a  \prec Z^{-1} \bigr) * Z \nonumber\\
 &=&  a + \lambda c  \prec \Bigl( Y * \bigl( Y^{-1} \succ a  \prec Z^{-1} \bigr)*Z\Bigr) 	+ \Bigl( a  \prec Z^{-1} 
	+ \bigl( \lambda c  \prec  Y\bigr)  \prec  \bigl( Y^{-1} \succ a  \prec Z^{-1} \bigr)\Bigr) \succ Z\\
 &=&  a + \lambda c  \prec \Bigl( Y * \bigl( Y^{-1} \succ a  \prec Z^{-1} \bigr)*Z\Bigr)  \\
	&& \hspace{5cm}
 	+ \Bigl( a  \prec Z^{-1}
	+ \bigl( \lambda c  \prec  Y\bigr)  \prec  \bigl( Y^{-1} \succ a  \prec Z^{-1} \bigr)\Bigr) \succ ( \lambda Z \succ b)\\       
&=&  a + \lambda c  \prec \Bigl( Y * \bigl( Y^{-1} \succ a  \prec Z^{-1} \bigr)*Z\Bigr)  \\
	&& \hspace{5cm}
	    + \lambda \Bigl( \bigl(a  \prec Z^{-1} \bigr)*Z           + \Bigl( \bigl( \lambda c  \prec  Y\bigr)  \prec  \bigl( Y^{-1} \succ a  \prec Z^{-1} \bigr)\Bigr)*Z\Bigr)  \succ b\\      
&=&  a + \lambda c  \prec \Bigl( Y * \bigl( Y^{-1} \succ a  \prec Z^{-1} \bigr)*Z\Bigr) 
		+ \lambda \Bigl( \bigl(a  \prec Z^{-1} \bigr)*Z  
          	+ \Bigl( Y \prec  \bigl( Y^{-1} \succ a  \prec Z^{-1} \bigr)\Bigr)*Z\Bigr)  \succ b\\          
&=&  a + \lambda c  \prec \Bigl( Y * \bigl( Y^{-1} \succ a  \prec Z^{-1} \bigr)*Z\Bigr) 
	+ \lambda  \Bigl( \bigl(a  \prec Z^{-1} \bigr)*Z  
	+ \Bigl( Y * \bigl( Y^{-1} \succ a  \prec Z^{-1} \bigr)\Bigr)*Z\Bigr)  \succ b\\   
&  & \hspace{5cm}    - \lambda \Bigl( \bigl( (Y * Y^{-1}) \succ (a  \prec Z^{-1}) \bigr)*Z\Bigr)  \succ b\\ 
&=&  a + \lambda c  \prec \Bigl( Y * \bigl( Y^{-1} \succ a  \prec Z^{-1} \bigr)*Z\Bigr) +
          \lambda \Bigl( Y * \bigl( Y^{-1} \succ a  \prec Z^{-1} \bigr)*Z\Bigr)  \succ b\\
&=& a + \lambda c \prec X +  \lambda X \succ b      \nonumber
\end{eqnarray*}}
\end{proof}

\begin{rmk}{\rm{
Observe that for $b,c \in A$, when using $(x \succ y) \prec z = x \succ (y \prec z)$, which implies  $ \un \succ a \prec \un = a$ for $a \in A$, we find:
\begin{eqnarray*}
	X &=& Y * \bigl( Y^{-1} \succ a  \prec Z^{-1} \bigr) * Z\\ 
	    &=& Y * \bigl( a  - \lambda (Y^{-1} \succ c) \succ a - \lambda a  \prec (b\prec Z^{-1})
                   +  \lambda^2 (Y^{-1} \succ c) \succ a  \prec (b\prec Z^{-1}) \bigr) * Z.
\end{eqnarray*}
This last expression reduces to $X= Y*Z$, the solution of (\ref{eq:degree1relation}), for $a = \un$. It is tempting in view of this observation to set $\un\succ\un\prec\un=\un$, but it would easily lead to some contradictions. Indeed, if we were able to define $\un\prec\un$ or $\un\succ\un$ in a way compatible with the dendriform axioms, an obvious computation leads
to $\un\succ\un\prec\un=0$! It is then safer to get rid of all expressions involving the undefined $\un\prec\un$ or $\un\succ\un$ in the calculations.
}}\end{rmk}

\begin{cor}\label{cor:pre-LieIVP} Let $\overline{A}$ be a unital dendriform algebra and $a,b \in A$.  Let $Z=Z(b)$ be a solution to $ Z= \un + \lambda  Z \succ b$. The solution to the ``left'' pre-Lie equation in  $\overline{A}$:
\begin{equation}
\label{eq:pre-Lie}
     X = a +   \lambda X \rhd b,
\end{equation}
is given by:
\begin{equation}
\label{eq:solPre-Lie}
	X = Z^{-1} * \bigl( Z \succ a  \prec Z^{-1} \bigr) * Z.
\end{equation}
\end{cor}

Recall that the solution $Z(b)$ of the equation of degree $(1,0)$, $Z= \un + \lambda  Z \succ b$ is a group-like element with respect to the coproduct $\Delta$, see  (\ref{Hopf}). We may therefore write the above solution (\ref{eq:solPre-Lie}) as:
$$
	X=Ad^{*}_{Z^{-1}}( Z \succ a  \prec Z^{-1} ),
$$
where as usual $Ad^{*}_{F}(G):=F*G*F^{-1}$. Due to the second dendriform axiom (\ref{A2}) we may interpret the term 
$Z \succ a  \prec Z^{-1}$ as a second group action, which we denote by:
\begin{equation}
\label{2ndgroupact}
	\Theta_Z(a):=Z \succ a  \prec Z^{-1}.
\end{equation}

\begin{lem} \label{lem:2ndgroupact} Let $\overline{A}$ be a unital dendriform algebra and $a,b \in A$. Let $A,B$ be two solutions of the equation of degree $(1,0)$.
\begin{enumerate}

\item $\Theta_A \circ \Theta_B (x) = \Theta_{A*B}(x)$

\item $\Theta_A \circ \Theta_{A^{-1}} (x) = \Theta_{A^{-1}} \circ \Theta_{A} (x)=x$.

\end{enumerate}

\end{lem}

\begin{proof} 
The second property follows from the first. Let us verify the first identity:
  \allowdisplaybreaks{
\begin{eqnarray*}
	\Theta_A \circ \Theta_B (x) &=& A \succ (B \succ x  \prec B^{-1})  \prec A^{-1}\\
							  &=& (A*B) \succ x  \prec (A*B)^{-1} = \Theta_{A*B}(x). 
\end{eqnarray*}}
This follows immediately from the dendriform axioms.
\end{proof}

We remark that the two group actions in general do not commute, i.e. $Ad^*_A \circ \Theta_B \neq \Theta_B \circ Ad^*_A$.

Recall that the Dynkin operator is the linear endomorphism of the tensor algebra $T(X)$ over an alphabet $X=\{x_1,\ldots
,x_n,\ldots\}$ into itself the action of which on words $y_1 \ldots y_n,\ y_i\in X$ is given by the left-to-right iteration of
the associated Lie bracket:
$$
    D(y_1,\dots,y_n) = [\cdots[[y_1, y_2], y_3]\cdots\!, y_n],
$$
where $[x,y]:=xy-yx$. The Dynkin operator is a quasi-idempotent: its action on a homogeneous element of degree~$n$ satisfies $D^2=nD$. The associated projector $D/n$ sends $T_n(X)$, the component of degree~$n$ of the tensor algebra, to the component of degree~$n$ of the free Lie algebra over $X$. The tensor algebra is a graded connected cocommutative Hopf algebra, and it is natural to extend the definition of $D$ to any such Hopf algebra as the convolution product of the antipode $S$ with the grading operator $N$: $D:=S \star N$ \cite{patreu2002, eunomia2006, EGP07, EMP07}. This applies in particular in the dendriform context to the Hopf algebra $H$ introduced above. We will write $D_n$ for $D\circ p_n$, where $p_n$ is the canonical projection from $T(X)$ (resp. $H$) to $T_n(X)$ (resp. $H_n$).

\begin{lem} \cite{EMP07} \label{lem:dynkin-prelie}
For any integer $n\ge 1$ and for any $x \in A$ we have:
\begin{equation}\label{eq:dynkin-prelie}
    D(w_{\succ}^{(n)}(x)) = {\ell}^{(n)}(x).
\end{equation}
\end{lem}

Let us come back to equation~(\ref{eq:degree1relation}) respectively Corollary~\ref{cor:pre-LieIVP}. Then, using $D(\un)=0$, we see immediately that the solution $Z=Z(a)$ to the equation $ Z= \un + \lambda  Z \succ a$ is mapped to the solution 
$$
	X := \lambda^{-1}D(Z) = Ad^*_{Z^{-1}}\circ \Theta_Z(a) =Z^{-1} * \bigl( Z \succ a  \prec Z^{-1} \bigr) * Z 
$$ 
of the corresponding pre-Lie equation~(\ref{eq:pre-Lierelation}) for $a=b=-c \in A$:
\begin{equation*}
	X = a + \lambda X \rhd a. 	
\end{equation*}


\section{The pre-Lie Magnus expansion}
\label{sect:preLieMag}

In this section we recall results from \cite{EM} where we obtained a recursive formula for the logarithm of the solutions of the two dendriform equations $X = \un + \lambda a \prec X$ and $Y = \un - \lambda Y \succ a$ in $A[[\lambda]]$, $a \in A$, of degree $(0,1),\ (1,0)$, respectively. 

Let us introduce the following operators in $(\overline A,\prec,\succ)$, where $a$ is any element of $A$:
 \allowdisplaybreaks{
\begin{eqnarray*}
    L_\prec[a](b)&:=& a \prec b \hskip 8mm L_\succ[a](b):= a\succ b \hskip 8mm
    R_\prec[a](b):= b \prec a \hskip 8mm R_\succ[a](b):= b\succ a\\
    L_\lhd[a](b)&:=& a \lhd b \hskip 8mm L_\rhd[a](b) := a \rhd b \hskip 8mm
    R_\lhd[a](b):=b \lhd a \hskip 8mm R_\rhd[a](b):= b \rhd a
\end{eqnarray*}}

In the following theorem we find a recursive formula for the logarithm of the solutions of the equations in (\ref{eq:relationYZ}). We will call it the {\sl{pre-Lie Magnus}} expansion for reasons which will become clear in Remark \ref{rmk:pre-LieMagnus}.  

\begin{thm} \cite{EM} \label{thm:mainMAGNUS} {\rm{(pre-Lie Magnus expansion)}}
Let $\Omega':=\Omega'(\lambda a)$, $a \in A$, be the element of $\lambda \overline{A}[[\lambda]]$ such that $Y=\exp^*(\Omega')$ and $Y^{-1}=\exp^*(-\Omega')$, where $Y$ and $Y^{-1}$ are the solutions of the two equations $Y = \un + \lambda a  \prec Y $ and $Y^{-1} = \un - \lambda Y^{-1} \succ a$ of degree $(0,1),\ (1,0)$, respectively. This element obeys the following recursive equation:
 \allowdisplaybreaks{
\begin{eqnarray}
    \Omega'(\lambda a) &=& \frac{R_\lhd[\Omega']}{1-\exp(-R_\lhd[\Omega'])}(\lambda a)=
                \sum_{m \ge 0}(-1)^m \frac{B_m}{m!}R^{(m)}_\lhd[\Omega'](\lambda a),\label{main1}
\end{eqnarray}}
or alternatively:
 \allowdisplaybreaks{
\begin{eqnarray}
    \Omega'(\lambda a) &=& \frac{L_\rhd[\Omega']}{\exp(L_\rhd[\Omega'])-1}(\lambda a)
            =\sum_{m\ge 0}\frac{B_m}{m!}L^{(m)}_\rhd[\Omega'](\lambda a),\label{main3}
\end{eqnarray}}
where the $B_l$'s are the Bernoulli numbers.
\end{thm}

Returning to Theorem~\ref{thm:dendsolution1} we see that the solution of equation  (\ref{eq:degree1relation}) writes:
$$
	X= \exp^*\bigl(\Omega'(c)\bigr)*\bigl(\exp^*\bigl(-\Omega'(c)\bigr)
	\succ a \prec 
	\exp^*\bigl(\Omega'(-b)\bigr)\bigr)*\exp^*\bigl(-\Omega'(-b)\bigr)
$$
and the solution to the pre-Lie recursion, i.e. equation~(\ref{eq:degree1relation}) for $c=-b$:
$$
	X = Ad^*_{\exp^*(-\Omega'(-b))} \circ \Theta_{\exp^*(\Omega'(-b))}(a). 
$$

\begin{rmk}\label{rmk:pre-LieMagnus}{\rm{ (classical Magnus expansion)
Let $\mathcal F$ be any algebra of operator-valued functions on the real line, e.g. smooth $n \times n$ matrix-valued functions, closed under integrals $I(U)(x):=\int_0^x U(s) ds$. Then, $\mathcal F$ is a dendriform algebra for the operations:
$$
    A \prec B (x):=A(x)\cdot \int\limits_0^xB(t)dt
 \qquad
    A \succ B (x):=\int\limits_0^xA(t)dt\cdot B(x)
$$
with $A,B \in \mathcal F$. This is a particular example of a dendriform structure arising from a Rota--Baxter algebra structure --of weight zero \cite{Aguiar00}. We refer to Section \ref{sect:appl} for further details on Rota--Baxter algebras and their connections to dendriform algebras. Here, let us simply mention that the Rota--Baxter operator on $\mathcal F$ giving rise to the dendriform structure is the Riemann integral map: $I(A)(x):=\int_0^xA(t)dt$. In this particular example we see that the pre-Lie Magnus expansion (\ref{main1}) with the pre-Lie product $L_{\rhd}[A](B)=(A\rhd B)(t)=[I(A),B](t)$:  
\begin{equation}
\label{eq:clMagnus}
	\Omega'(A)(s)=\dot \Omega(s) = \sum_{m\ge 0}\frac{B_m}{m!}ad_{\Omega(s)}^{(m)}(A(s)),  
\end{equation}
coincides with Magnus' logarithm \cite{Magnus}, $\Omega(t)=\log(Y(t))$, of the solution of the initial value problem $\dot{Y}(t)=A(t)Y(t)$, $Y(0)=Y_0=1$, such that:
$$
    Y(t) = \exp\Bigl( \int^t_0 \dot{\Omega}(s)\, ds \Bigr)
         = \sum_{n \ge 0} \frac{\Omega(t)^n}{n!},\qquad \Omega(0)=0,
$$
for more details on Magnus' work see also \cite{CasasIserles06,Gelfand,IsNo99,KlOt,MielPleb,Strichartz,Wilcox,Zanna}. By now we hope the above used prime notation, $\Omega'$, in equation (\ref{main1}) and (\ref{main3}) of Theorem \ref{thm:mainMAGNUS} has become self-evident. Notice that if we suppose the algebra $\mathcal F$ to be unital, the unit (which we denote by $1$) has nothing to do with the added unit $\un$ of the underlying dendriform algebra. In fact, we extend the Rota--Baxter algebra structure to $\overline {\mathcal{F}}$ by setting:
\begin{equation}
\label{unite}
    R(\un):=1,\hskip 8mm \wt R(\un):=-1 \hskip 5mm \hbox{ and
    }\un.x=x.\un=0\
    \hbox{ for any } x\in\overline {\mathcal{F}}.
\end{equation}
This is consistent with the axioms (\ref{unit-dend}) which in particular yield $\un\succ x=R(\un)x$ and $x\prec \un=-x\wt R(\un)$. Later, in Section \ref{sect:appl} we will provide more details in the context of applications.}}
\end{rmk}

Returning to the two group actions, we describe the infinitesimal actions. Let $\overline{A}$ be a unital dendriform algebra. Let $Y=\exp^*(\Omega'(\lambda a))$, $a \in A$, be the solution of the equation $Y = \un + \lambda a \prec Y$  of degree $(0,1)$ in $\overline{A}[[\lambda]]$.
\begin{enumerate}

\item $\Theta_Y (x) = Y \succ x \prec Y^{-1} \sim x + \lambda a \rhd x + O(\lambda^2)$

\item $Ad^*_Y (x) = Y * x * Y^{-1} \sim x + \lambda [a,x] + O(\lambda^2)$.

\end{enumerate}


\section{The pre-Lie Fer expansion}
\label{sect:preLieFer}

Following Fer's original work~\cite{Fer}, see also \cite{Iserles84,Iserles00,Wilcox}, we now make a simple Ansatz for the solution of the degree $(0,1)$ equation $Y=\un + \lambda a \prec Y$ :
$$
    Y=\exp^*(\lambda a)*V_1
$$
and return this into the recursion. This then leads to what we named the {\sl{pre-Lie Fer}} expansion.

\begin{thm} \cite{EM} \label{thm:mainFer} {\rm{(pre-Lie Fer expansion)}}
Let $(\overline A,\prec,\succ)$ be a dendriform algebra augmented by a unit $\un$~(\ref{unit-dend}). Let $U'_0:=\lambda a$,
$U'_n:=U'_n(a)$, $n \in \mathbb{N}$, $a \in A$, be elements in $\lambda \overline{A}[[\lambda]]$, such that
$$
    Y = \overrightarrow{\prod\limits_{n \ge 0}}^* \exp^*(U'_n)
    \qquad\
    Y^{-1} = \overleftarrow{\prod\limits_{n \ge 0}}^* \exp^*(-U'_n)
$$
where $Y$ and $Y^{-1}$ are solutions of $Y = \un + \lambda a  \prec Y $ and $Y^{-1} = \un - \lambda Y^{-1} \succ a$, respectively. Then these elements $U'_n$ obey the following recursive equation:
 \allowdisplaybreaks{
\begin{eqnarray*}
    U'_{n+1}:= \sum_{l>0} \frac{(-1)^l l}{(l+1)!} L_\rhd[U'_n]^{(l)}(U'_n),\quad\ n \ge 0.
\end{eqnarray*}}
\end{thm}

The presentation given here in the context of dendriform algebras, using the corresponding pre-Lie products, reduces to Fer's classical expansion \cite{Fer}, see also \cite{MZ}, when working in the dendriform algebra introduced in Remark~\ref{rmk:pre-LieMagnus}. 

 
\section{An equation of degree $(2,0)$}
\label{sect:secondorder}

Let $D$ be a dendriform algebra and $\overline D$ its augmentation by a unit $\un$, see (\ref{unit-dend}). We are interested in solving the following equation of degree $(2,0)$ in $\overline D[[\lambda]]$: 
\begin{equation}
\label{eq:order2}
	X= d +\lambda X\succ c+\lambda^2(X\succ b)\succ a,
\end{equation}
for $a,b,c \in D$ and $d \in \overline D$. We introduce matrix dendriform
algebras, in which equation \eqref{eq:order2} will be transformed into a
system of order $(1,0)$ equations.


\subsection{Matrix dendriform algebras}
\label{ssect:matrixDend}

Let $(D,\prec,\succ)$ be a dendriform algebra. Then the space $\Cal M_n(D)$ of square $n\times n$-matrices with entries in $D$ is a dendriform algebra, with operations defined by:
\begin{equation*}
	(a\prec b)_{ij}:=\sum_{i=1}^na_{ik}\prec b_{kj},\hskip 12mm
	(a\succ b)_{ij}:=\sum_{i=1}^na_{ik}\succ b_{kj}.
\end{equation*}
We can indeed verify the first dendriform axiom:
 \allowdisplaybreaks{
\begin{eqnarray*}
	[(a\prec b)\prec c]_{ij} &=&\sum_{k=1}^n(a\prec b)_{ik}\prec c_{kj}\\
						&=&\sum_{k,l=1}^n(a_{il}\prec b_{lk})\prec c_{kj}\\
						&=&\sum_{k,l=1}^na_{il}\prec(b_{lk}*c_{kj})=\sum_{l=1}^na_{il}\prec(b* c)_{lj} =[a\prec(b*c)]_{ij},
\end{eqnarray*}}
hence $(a\prec b)\prec c=a\prec(b*c)$ in $\Cal M_n(D)$, and the two other axioms can be checked similarly. If $\overline D:=D \oplus k.\un$ (resp. $\overline{\Cal M_n(D)}:=\Cal M_n(D)\oplus k.\un_n$) is the dendriform algebra $D$ (resp. $\mathcal{M}_n(D)$) augmented with a unit, then we can identify the unit $\un_n$ with the diagonal matrix with dendriform units $\un$'s on the diagonal and $0$ elsewhere.


\subsection{Transformation into a dendriform system of degree $(1,0)$}
\label{ssect:system}

First, we will show how to transform equation (\ref{eq:order2}) for $d=\un$ into a ``first-order" equation in $\Cal M_2(D)[[\lambda]]$. Observe the size of the matrices, i.e. $2=1+(2-1)2/2$.

\begin{lem} \label{base}
Let $Z \in \overline{\Cal M_2(D)}[[\lambda]]$ be the solution of the ``matrix dendriform" equation of degree $(1,0)$:
\begin{equation}
\label{eq:mat2order1}
	Z = \un_2 + \lambda Z \succ M.
\end{equation}
Then for any $v=(v_1,v_2)\in k^2$ the solution in $(\overline D[[\lambda]])^2$ of the equation:
\begin{equation}
\label{eq:vect-order1}
	Y=v\un_2+\lambda Y\succ M,
\end{equation}
with $Y\succ M:=(Y_1\succ M_{11}+Y_2\succ M_{21},\,Y_1\succ M_{12}+Y_2\succ M_{22})$ is given by: 
$$
	vZ:=(v_1Z_{11}+v_2Z_{21},\,v_1Z_{12}+v_2Z_{22}).
$$
\end{lem}

{\bf Proof}: We immediately check:
 \allowdisplaybreaks{
\begin{eqnarray*}
	   vZ &=& v\un_2+\lambda v(Z\succ M)\\
	 	 &=& v\un_2+\lambda(vZ)\succ M.
\end{eqnarray*}}
\qed

\begin{prop}\label{prop:matSOL}
Let $X$ be the solution of equation (\ref{eq:order2}) in $\overline D[[\lambda]]$ with $d=\un$. Then the vector $Y:=(X,\,\lambda X\succ b)$ in $(\overline D[[\lambda]])^2$ is the solution of equation (\ref{eq:vect-order1}) with $v=(1,0) \in k^2$ and:
\begin{equation*}
	M=\begin{pmatrix}
		c & b\\
		a & 0
	     \end{pmatrix}.
\end{equation*}
\end{prop}

{\bf Proof}: One easily checks that:
 \allowdisplaybreaks{
\begin{eqnarray*}
	v\un_2+\lambda Y\succ M&=&(\un,\,0)+\lambda(X,\,\lambda X\succ b)\succ
		\begin{pmatrix}
			c&b\\
			a&0
		\end{pmatrix}\\
	&=&(\un,\,0)+\lambda\bigl( X\succ c+\lambda(X\succ b)\succ a,\,X\succ b\bigr)\\
	&=&(\un,\,0)+(X-\un,\,\lambda X\succ b)\\
	&=&Y.
\end{eqnarray*}}
\qed

Using Theorem \ref{thm:mainMAGNUS}, i.e. the pre-Lie Magnus expansion, we find immediately the solution to (\ref{eq:mat2order1}):

\begin{cor} \label{cor:matsolMAGNUS}
Let $\Omega'=\log^*Z \in \Cal M_2(D[[\lambda]])$, where $Z$ is the solution of equation (\ref{eq:mat2order1}) with
$M=\begin{pmatrix}c&b\\ a&0\end{pmatrix}$, given by the pre-Lie Magnus expansion:
\begin{equation*}
	\Omega'(\lambda M) = \lambda M - \frac 12 \lambda^2 M \rhd M 
							+ \frac14 \lambda^3 ( M \rhd M ) \rhd M + \lambda^3 \frac{1}{12}  M \rhd (M \rhd M) \cdots.
\end{equation*}
Then the solution $X$ of equation (\ref{eq:order2}), $d=\un$, is such that the line vector $(X,\,\lambda X\succ b)$ is the first line of the matrix $Z=\exp^*(\Omega')$, i.e.: 
$$  
	X=(1,0)Z(1,0)^{\rm{t}}.
$$
\end{cor}

The following corollaries are readily verified. 

\begin{cor} \label{cor:dcbaCase1}
Let $X$ be the solution of equation (\ref{eq:order2}) in $\overline D[[\lambda]]$. Then the vector $Y:=(X,\,\lambda X\succ b)$ in $(\overline D[[\lambda]])^2$ is the solution of equation: 
\begin{equation}
\label{eq:vect-order1-d}
	Y=v (d * \un_2) + \lambda Y \succ M,
\end{equation} 
with $v=(1,0) \in k^2$ and $M$ as in Proposition \ref{prop:matSOL}. 
\end{cor}

Here $d * \un_2$ is the diagonal matrix in $\Cal M_2(D)$ with $d$ on its diagonal and $0$ else. Using Lemma \ref{lem:m=1eq} we immediately get:

\begin{cor} \label{cor:dcbaCase2}
Let $\Omega'=\log^*Z \in \Cal M_2(D[[\lambda]])$, where $Z$ is the solution of equation (\ref{eq:mat2order1}) with $M$ as in Proposition \ref{prop:matSOL}, given by the pre-Lie Magnus expansion:
\begin{equation*}
	\Omega'(\lambda M) = \lambda M - \frac 12 \lambda^2 M \rhd M 
							+ \frac14 \lambda^3 ( M \rhd M ) \rhd M + \lambda^3 \frac{1}{12}  M \rhd (M \rhd M) \cdots.
\end{equation*}
The line vector $(X,\,\lambda X\succ b)$ is the first line of the matrix $U:=\bigl(d\un_2\prec\exp^*(-\Omega')\bigr)* \exp^*(\Omega')$.
\end{cor}

The equation of degree $(0,2)$ in $\overline{D}[[\lambda]]$: 
\begin{equation}
\label{eq:order2S}
	\hat X = \un + \lambda\  c \prec  X + \lambda^2\   a \prec ( b \prec  \hat X).
\end{equation}
can be treated similarly:
 
\begin{prop}\label{prop:matSOLtrans}
Let $ \hat X$ be the solution of equation (\ref{eq:order2S}) in $\overline D[[\lambda]]$. Then the column vector $  \hat Y:= ( \hat X,\,\lambda\  b \prec  \hat X)^t$ in $(\overline D[[\lambda]])^2$ is a solution of equation: 
\begin{equation}
\label{eq:vect-order1trans}
	 \hat Y=\un_2 v +\lambda\   M^t \prec  \hat Y,
\end{equation}
with $v=(1,0)^t \in k^2$ and:
\begin{equation*}
	M^t= \begin{pmatrix}
 		c &  a\\
 		b & 0
		 \end{pmatrix}.
\end{equation*}
\end{prop}

\begin{proof}
The solution of equation \eqref{eq:order2S} in $\overline D[[\lambda]]$ is exactly the solution of equation \eqref{eq:order2} in $\overline {D'}[[\lambda]]$ where $D'$ is the dendriform algebra $(D,\preceq,\succeq)$, with $x\preceq y:=y\succ x$ and $x\succeq y:=y\prec x$.
\end{proof}


\section{Higher order equations}
\label{sect:higherorder}

The general degree $(m,0)$ equation writes: 
\begin{equation}\label{m-zero}
     X  =  a_{00} +  \sum_{q=1}^{m} \lambda^{q}  \omega^{(q+1)}_\succ(X,a_{q1},\ldots, a_{qq}).
\end{equation}
The cases $m=3$ and $m=4$ can be worked out by hand and will give a hint for the general case. Indeed, the length $4$ line vector:
\begin{equation*}
	Y_3:=\bigl( X,\,\lambda X\succ a_{21},\,\lambda X\succ a_{31},\,\lambda^2(X\succ a_{31})\succ a_{32}\bigr)
\end{equation*}
is the solution of the equation:
\begin{equation*}
Y_3=(a_{00},0,0,0)+\lambda Y_3\succ M_3,\hbox{ with }
	M_3=\begin{pmatrix}
			a_{11}&a_{21}&a_{31}&0\\
			a_{22}&0        &0         &0\\
				 0&0        &0         &a_{32}\\
		      a_{33}&0        &0         &0
\end{pmatrix}.
\end{equation*}
The length $7$ line vector:
{\small{
\begin{equation*}
	Y_4:=\Big(X,\,\lambda X\succ a_{21},\,\lambda X\succ a_{31},
	\,\lambda^2(X\succ a_{31})\succ a_{32},
	\,\lambda X\succ a_{41},\,\lambda^2(X\succ a_{41})\succ a_{42},
	\,\lambda^3\bigl( (X\succ a_{41})\succ a_{42}\bigr)\succ a_{43}\Big)
\end{equation*}}}
is the solution of the equation:
\begin{equation*}
	Y_4=(a_{00},0,0,0,0,0,0)+\lambda Y_4\succ M_4,\hbox{ with }
	M_4= \begin{pmatrix}
			a_{11}&a_{21}&a_{31}&0        &a_{41} &0         &0\\
			a_{22}&0        &0         &0        &0          &0         &0\\
				 0&0         &0        &a_{32}&0          &0         &0\\
		      a_{33}&0		  &0		&0	      &0	      &0          &0\\
				0&0		  &0		&0	      &0	      &a_{42} &0\\
				0&0		  &0		&0	      &0	      &0          &a_{43}\\
		     a_{44}&0		 &0 		&0	      &0	      &0           &0
\end{pmatrix}.
\end{equation*}
We recognize the matrix $M_3$ as the upper left $4\times 4$-submatrix of $M_4$. The solution of the general equation \eqref{m-zero} will be the first coefficient of a vector $Y_m$ of length $1+m(m-1)/2$, whose coefficients (discarding the first one) are given by $\lambda^j \omega_\succ^{j+1}(X,a_{q1},\ldots,a_{qj}),\, 1\le j<q\le m$. They are displayed according to the lexicographic order of the pairs $(q,j)$ indexing them. The vector $Y_m$ is the solution of the equation
\begin{equation*}
	Y_m=(a_{00},\!\!\mathop{\underbrace{0,\ldots ,0}}_{\frac{m(m-1)}{2} \hbox{ \eightrm times}}\!\!)+\lambda Y_m\succ M_m,
\end{equation*}
where the square matrix $M_m$ of size $1+m(m-1)/2$ is recursively defined by:
\begin{equation}
	M_m=
		\begin{pmatrix}
			M_{m-1} & A_1\\
			A_2		& A_3\\
		\end{pmatrix}=
	\begin{pmatrix}
		M_{m-1}& \hskip -2mm
		\begin{pmatrix}
		a_{m1}&\ 0\ &\hskip 2mm 0\hskip 2mm &\cdots &\hskip 6mm 0\hskip 6mm\\
			  0&0&0&\cdots &0\\
		\vdots &\vdots&\vdots&\cdots&\vdots\\
			 0&0&0&\cdots &0
		\end{pmatrix}\\
&\\
		\begin{pmatrix}
			0&0&\cdots &0\\
			0&0&\cdots &0\\
	     \vdots &\vdots&\cdots&\vdots\\
		      0&0&\cdots &0\\
	  a_{mm}&0&\cdots &0
		\end{pmatrix}&
		\begin{pmatrix}
		  0&a_{m2}&0          &\cdots & 0\\
		  0&		  0&a_{m3}&\cdots & 0\\
	 \vdots &\vdots   &\ddots  &\ddots & \vdots\\
		 0&           0&          0&\cdots & a_{m(m-1)}\\
		 0&		  0&          0&\cdots &0
		\end{pmatrix}
\end{pmatrix}.
\end{equation}
We proceed then as in paragraph \ref{ssect:system}: according to Lemma \ref{lem:m=1eq}. The line vector $Y_m$ is the first line of the matrix of size $N:=1+m(m-1)/2$ defined by $\bigl( (a_{00}*\un_{N})\prec Z^{*-1}\bigr)*Z$, where $Z$ is the solution of the order $(1,0)$ matrix dendriform equation:
\begin{equation}\label{genmatrix}
	Z = \un_{N} + Z\succ M_m.
\end{equation}
In the simpler case when $a_{00}=\un$ instead of being an element of $D$, the line vector $Y_m$ is just the first line of the solution $Z$ of equation \eqref{genmatrix}.\\

The solution of general equation of order $(0,n)$ can then be derived immediately, by the same trick as in the proof of Proposition  \ref{prop:matSOLtrans}. Higher-order equations of type \eqref{maxEq} of order $(m,n)$ remain to be solved.


\section{Applications}
\label{sect:appl}

For completeness let us recall the recursions and its solutions we described in the foregoing sections. Let $(\overline{D}[[\lambda]],\prec,\succ)$ be a unital dendriform algebra, $a,b \in D$. The principal equations are:
$$
	Y= \un + \lambda a \prec Y  \qquad\qquad\ 	Z = \un - \lambda Z \succ b,   
$$
of degree $(0,1)$ and $(1,0)$, respectively, and with formal solutions $Y=Y(a)=\sum_{n\ge0} \omega^{(n)}_{\prec}(\lambda a)$ and $Z=Z(b)=\sum_{n\ge 0} \omega_{\succ}^{(n)}(- \lambda b)$ in $\overline{D}[[\lambda]]$. Recall that $Z(a)=Y^{-1}(a)$ (\ref{eq:relationYZinverse}). In Theorem \ref{thm:mainMAGNUS} (pre-Lie Magnus expansion) we showed that $Y(a)=\exp^*(\Omega'(\lambda a))$, whereas Theorem \ref{thm:mainFer} (pre-Lie Fer expansion) concludes that $Y(a)=\overrightarrow{\prod}_{n \ge 0}^* \exp^*(U'_n(a))$, $U'_0=\lambda a$. Another pair of recursions in $\overline{D}[[\lambda]]$ are: 
$$
	E= a + \lambda b \prec E  	\qquad\qquad\ 		F = c + F \succ \lambda d,   
$$
$a,b,c,d \in D$. Lemma \ref{lem:m=1eq} provides the solutions to these equations in $\overline D[[\lambda ]]$. They can be expressed in terms of $Y=Y(b)$ and $Z=Z(d)$, i.e. $E =  Y * \bigl( Y^{-1} \succ  \lambda a\bigr)$ respectively $F = \bigl(  \lambda c \prec Z^{-1} \bigr) * Z $. Finally, putting together the former two recursions we arrive at the equation of degree $(1,1)$:
$$
	X = a + \lambda X \succ b + \lambda c \prec X, 
$$  
$a,b,c \in D$, and its solution in terms of $Y=Y(c)$ and $Z=Z(b)$:
\begin{equation*}
	X = Y * \bigl( Y^{-1} \succ a  \prec Z^{-1} \bigr) * Z. 
\end{equation*}
For $c=-b \in D$ we arrive at the pre-Lie recursion in $\overline D$: 
$$
	X = a + X \rhd b,
$$
and its solution in terms if $Z=Z(-b)$:
$
	X= Ad^*_{Z^{-1}}\circ \Theta_{Z}(a).
$

Recall from the introduction the link between recursions in dendriform algebras and simple initial value problems, say, for instance in  algebras of matrix valued functions closed under the Riemann integral map. In fact, the link here is provided by the relation between associative Rota--Baxter and dendriform algebras.


\subsection{Rota--Baxter algebras}
\label{ssect:RB}

Recall \cite{Baxter,E,Rota} that an associative Rota--Baxter algebra (over a field $k$) is an associative $k$-algebra $A$ endowed with a $k$-linear map $R: A \to A$ subject to the following relation:
\begin{equation}\label{RB}
    R(a)R(b) = R\bigl(R(a)b + aR(b) + \theta ab\bigr).
\end{equation}
where $\theta \in k$. The map $R$ is called a {\sl Rota--Baxter operator of weight $\theta$\/}. The map $\widetilde{R}:=-\theta id-R$ also is a weight $\theta$ Rota--Baxter map. Both the image of $R$ and $\tilde{R}$ form subalgebras in $A$. Associative Rota--Baxter algebras arise in many mathematical contexts, e.g. in integral and finite differences calculus, but also in perturbative renormalization in quantum field theory~\cite{egm}.

A few examples are in order. First, recall the classical integration by parts rule showing that the ordinary Riemann integral is a weight zero Rota--Baxter map. Moreover, on a suitable class of functions, we define the following Riemann summation operators:
\begin{eqnarray}
    R_\theta(f)(x) := \sum_{n = 1}^{[x/\theta]} \theta f(n\theta)
    \qquad\ {\rm{and}} \qquad\
    R'_\theta(f)(x) := \sum_{n = 1}^{[x/\theta]-1} \theta f(n\theta),
\label{eq:le-clou}
\end{eqnarray}
which satisfy the weight $-\theta$ and the weight $\theta$ Rota--Baxter relation, respectively.  The summation operator:
\begin{equation}
\label{summation}
	S(f)(x) := \sum_{n\geq 1} \theta f(x + \theta n).
\end{equation}
also satisfies relation (\ref{RB}) on a suitable class of functions. Another useful example is given in term of the $q$-dilatation operator, $\sigma_qf(x):=f(qx)$. In fact, one shows that  the $q$-summation operator:
\begin{equation}
\label{qsummation}
	S_q(f)(x) :=\sum_{n\geq 0} \sigma_q^nf(x) = \sum_{n\geq 0} f(q^nx)
\end{equation}
satisfies relation (\ref{RB}) on a suitable class of functions, with $\theta=-1$. In fact, one may think of (\ref{RB}) as a generalized integration by parts identity. Indeed, the reader will have no difficulty in checking duality of~(\ref{RB}) with the `skewderivation' rule~\cite{Kalliope}:
\begin{equation*}
    \partial(fg) = \partial(f)g + f\partial(g) + \theta \partial(f)\partial(g),
\end{equation*}
i.e. $R\partial x=x + c$ and $\partial R x = x$, $x \in A$ and $c \in \ker \partial$. For instance, in the case of the $q$-summation operator (\ref{qsummation}), we find $\partial_qf(x):=(id-\sigma_q)f(x)=f(x)-f(qx)$, satisfying:
 \allowdisplaybreaks{
\begin{eqnarray*}
	\partial_q(fg) &=& \partial_q(f)g + \sigma_q(f)\partial_qg\\
	                     &=& \partial_q(f)g + f\partial_q(g) - \partial_q(f)\partial_q(g) . 
\end{eqnarray*}
From this one can prove the Rota--Baxter property for the $q$-summation operator $S_q$ (weight $\theta=-1$), which is equivalent to the 
so called $q$-integration by parts rule:
$$
	S_q(f)S_q(g)=S_q(S_q(f)g) + S_q(\sigma_q(f)S_q(g)).
$$ 
We call a $k$-algebra $A$ with both, a Rota--Baxter map $R$ and its corresponding skewderivation $\partial$ a Rota--Baxter pair $(A,R,\partial)$. Although, we should underline that, in general, Rota--Baxter algebras do not come with such a skewderivation, e.g. let $A$ be a $\mathbb{K}$-algebra which decomposes directly into subalgebras $A_1$ and $A_2$, then the projection to $A_1$, $R: A \to A$, $R(a_1,a_2)=a_1$, is an idempotent Rota--Baxter operator, i.e. of weight $\theta=-1$, see~\cite{Kalliope,EM} for more details. 

\begin{prop}\label{RBtridend}\cite {E}
Any associative Rota--Baxter algebra gives a {\sl{tridendriform algebra}}, $(T_R,<,>,\bullet_\theta)$, in the sense that the Rota--Baxter structure yields three binary operations:
 \allowdisplaybreaks{
\begin{eqnarray*}
    a < b := aR(b),
    \hskip 8mm
    a > b := R(a)b,
    \hskip 8mm
    a \bullet_\theta b := \theta ab,
\end{eqnarray*}}
satisfying the tridendriform algebra axioms (\ref{DT}).
\end{prop}

The associated associative product $*_\theta$ is given by
 \allowdisplaybreaks{
\begin{eqnarray}
\label{RBassoprod}
    a *_\theta b := aR(b) + R(a)b + \theta ab.
\end{eqnarray}}
It is sometimes called the ``double Rota--Baxter product'', and verifies:
\begin{equation}
\label{RBhom}
    R( a *_\theta b) = R(a)R(b), \quad \widetilde{R}( a *_\theta b)=- \widetilde{R}(a)\widetilde{R}(b)
\end{equation}
which is just a reformulation of the Rota--Baxter relation (\ref{RB}). In fact, one easily shows that for any Rota--Baxter algebra $A$, the vector space $A$ equipped with the new product (\ref{RBassoprod}) is again a Rota--Baxter algebra, denoted $A_R$, though in general non-unital.    

From the general link between dendriform and tridendriform algebras it simply follows:

\begin{cor}\label{cor:RBdend}
Any associative Rota--Baxter algebra gives rise to a dendriform algebra structure, $(D_R,\prec,\succ)$, given by:
 \allowdisplaybreaks{
\begin{eqnarray}
\label{dendRB}
    a \prec b &:=& aR(b)+\theta ab =-a\widetilde{R}(b),\hskip 8mm a \succ b:=R(a)b.
\end{eqnarray}} 
\end{cor}

The dendriform pre-Lie products, (\ref{prelie1}) and (\ref{prelie2}), can be written explicitly:
$$
	a \rhd b = [R(a),b] - \theta ba \qquad\qquad\ a \lhd b = [a,R(b)] + \theta ab.
$$

Recall \cite{EM} that we extend the dendriform structure $(D_R,\prec,\succ)$ on the Rota--Baxter algebra $(A,R)$ to $\overline{D}_R$ by setting $R(\un):=1$, $\wt R(\un):=-1$, and $\un.x=x.\un=0$ for any $x\in\overline{D}_R$. This is consistent with the axioms (\ref{unit-dend}) which in particular yield $\un\succ x=R(\un)x$ and $x\prec \un=-x\wt R(\un)$.


\subsection{A link with generalized initial value problems}
\label{ssect:genIVP}

Let us go back to Remark \ref{rmk:pre-LieMagnus}. Assume for the moment that we work in a noncommutative unital function algebra $\mathcal{F}$, say, suitable $n \times n$-matrix valued functions, with a Rota--Baxter pair $(R,\partial)$. Let us consider the initial value problems (IVPs) 
\begin{eqnarray}
\label{eq:genIVP}
	\dot{Y}:= \partial Y = YB - CY,\;\; Y(0)=Y_0
\end{eqnarray}
with $B,C \in \mathcal{F}$. Iserles in~\cite {Iserles01} looked at this equation in the classical setting of Remark \ref{rmk:pre-LieMagnus}, where its straightforward solution is given in terms of the classical Magnus expansion (\ref{eq:clMagnus}):
\begin{eqnarray}
\label{eq:SolgenIVP}
	Y=-\exp\bigl(\Omega(C)\bigr)Y_0\exp\bigl(-\Omega(B)\bigr).
\end{eqnarray}
Here $U:=\exp(\Omega(B))$ and $V:=\exp(-\Omega(C))$ solve the IVPs, $\dot U=-UB$, $U(0)=1$ and $\dot V=CV$, $V(0)=1$, respectively. 

Coming back to equation (\ref{eq:genIVP}) in the context of a Rota--Baxter pair $(R,\partial)$ of weight zero, we may transform it into the recursion for $\dot Y$:
$$
	\dot{Y}= A + R(\dot Y) A - BR(\dot Y),\;\; A = Y_0B - CY_0,
$$ 
which now becomes a dendriform equation of degree $(1,1)$ for the element $X:=\dot Y$:
$$
	X = A + X \succ B +  C \prec X.
$$  
One verifies by using the relation between $\partial$ and $R$ as well as (\ref{RBhom}), that $R$ maps the solution $X$, see (\ref{eq:relation1Sol}), of this equation to (\ref{eq:SolgenIVP}), i.e. $R(X)=Y$. 

\smallskip

However, in the light of the fact that Rota--Baxter algebras in general may not come with a corresponding skewderivation, it is crucial to realize that $X$ is simply an element in $A_R$, the Rota--Baxter algebra with product (\ref {RBassoprod}). Indeed, we would like to emphasize that our findings in the context of dendriform algebras hold through for any associative Rota--Baxter algebra $(A,R)$.          

Hence, starting with a Rota--Baxter algebra $(A,R)$, Theorem \ref{thm:mainMAGNUS} (resp. Theorem \ref{thm:mainFer}) gives us the solution $\hat{Y}=\exp(R(\Omega'(a))$ to the equation $\hat{Y} = 1 + \lambda R(a\hat{Y}) $. We may interpret this as the "integral-like" equation corresponding to the IVP, $\dot Y = a Y$, $Y_0=1$. Though, we only used (\ref{RBhom}) as well as  (\ref{unit-dend}), and interpret $\dot Y$ purely as an element in $A_R$. Theorem \ref{thm:dendsolution1} then implies for $b=0$ the solution to recursion $\hat{E} = \lambda a + \lambda R(c\hat{E}) $, given by $\hat{E}=\exp(R(\Omega'(c))R(\exp(R(\Omega'(c))a)$, analogously for $c=0$. From a dendriform algebraic point of view the classical IVP (\ref{eq:genIVP}) should be replaced by an equation in terms of the pre-Lie product:
$$
	X=a + X \rhd b,
$$   
$a,b \in A$, where its solution  (\ref{eq:solPre-Lie}) again is an element in $A_R$:
$$ 
	X = Ad^*_{\exp^*(-\Omega'(-b))} \circ \Theta_{\exp^*(\Omega'(-b))}(a). 
$$
The formal limit of $\theta \to 0$ reduces to the Lie bracket flow like equation. It is important to remark that, as there are two natural Rota--Baxter maps on $A$, i.e. $R$ and $\tilde{R}$,  we can map $X \in A_R$ either via $R$ to $Y:=R(X) \in A$ or via $\tilde{R}$ to $\tilde{Y}:=\tilde{R}(X) \in A$.


\subsection{A link with the Riccati differential equation}
\label{ssect:Riccati}

The Riccati differential equation writes:
\begin{equation}
\label{riccati}
	\dot y = p + qy + ry^2,
\end{equation}
where $y,p,q$ and $r$ are differentiable functions on the real line. We suppose for convenience that $r$ never vanishes. Looking at the transformation:
\begin{equation}
\label{changefun}
	y = \frac{\dot w}{r w}
\end{equation}
one easily checks that the solutions of \eqref{riccati} are also solutions of the following linear homogeneous second-order differential equation:
\begin{equation}
\label{riccati2}
	r \ddot w - (\dot r+qr)\dot w + pr^2w = 0.
\end{equation}
Now consider equation \eqref{eq:order2}. Recall Remark~\ref{rmk:pre-LieMagnus}, in this particular situation, we find a dendriform algebra structure on the set $\mathcal{F}$ of twice differentiable functions on $\R$, which we denote by $D$, with dendriform left and right products:
\begin{equation}
	f \prec g:=f \cdot R(g)\qquad
	f \succ g:=R(f)\cdot g.
\end{equation}
Here $R=I$ is the Riemann integral operator on $\mathcal{F}$ defined by:
\begin{equation}
	R(f)(x):=\int_0^x f(t)\,dt.
\end{equation}
Recall from \cite{EMP07b, EM} that the operator $R$ (which is a Rota--Baxter map of weight zero) extends to $\overline{D}$ by putting $R(\un)=1$, where $1$ is the usual unit of the algebra of functions, i.e. the function equal to $1$ everywhere. The operator $R$ extends naturally to $\overline D[[\lambda]]$ by
$\lambda$-linearity. Starting from a solution $X$ of \eqref{eq:order2} and putting $y:=R(X)$ we have:
\begin{equation}
	y = 1+ \lambda R\bigl(R(X)c\bigr) + \lambda^2 R\Big(R\bigl(R(X)a\bigr)b\Big),
\end{equation}
and hence:
\begin{equation}
	\dot y=\lambda yc + \lambda^2R\bigl(ya\bigr)b.
\end{equation}
We suppose at this point that $b$ is a nowhere vanishing function on $\R$. Dividing out by $b$ and differentiating once again we obtain:
\begin{equation}
	\frac 1b \ddot y - \frac{\dot b}{b^2}\dot y = \lambda y \frac{d}{dt} \frac{c}{b} + \lambda  \frac{c}{b}  \dot y +  \lambda^2 ya,
\end{equation}
hence:
\begin{equation}
	b \ddot y - (\dot b + \lambda bc)\dot y -  (\lambda \frac{d}{dt} \frac{c}{b} + \lambda^2 a) b^2 y = 0.
\end{equation}
We recognize a formal version of equation \eqref{riccati2}, with $p=-(\frac{d}{dt} \frac{q}{b} + \lambda^2 a)$, $q=\lambda c$ and $r=b$.


\subsection{Vogel's identity}
\label{ssect:VogelsIdentity}

We present here a generalization of an operator identity and its solution, found by Vogel in \cite{Vogel} in the context of probability theory, to noncommutative associative Rota--Baxter algebras.  

Let $(A,R)$ be an associative Rota--Baxter algebra of weight $\theta \in \mathbb{K}$, $a,b,c \in A$. The following equation defined in $A[[\lambda]]$ appeared in \cite{Vogel}:    
\begin{equation}
\label{eq:Vogel}
	X = a + \lambda R(X)b + \lambda c\tilde{R}(X).
\end{equation} 
For $A$ being commutative Vogel found the solution, $X=B-C$, where $C=C(b,c)$ and $B=B(b,c)$ are defined in $A[[\lambda]]$ as follows:
\begin{eqnarray}
\label{eq:solVogel1}
	C &:=& \exp\Big(R\bigl(\alpha(\lambda)\bigr)\Big)R\Big(a \exp\bigl( R\bigl(\theta^{-1}\log(1+\lambda \theta c)\bigr) + \tilde{R}\bigl(\theta^{-1}\log(1+\lambda \theta b)\bigr)\bigr)\Big)\\
\label{eq:solVogel1}
	B &:=& \exp\Big(-\tilde{R}\bigl(\alpha(\lambda)\bigr)\Big)\tilde{R}\Big(a \exp\bigl(R\bigl( \theta^{-1}\log(1+\lambda \theta  c)\bigr) + \tilde{R}\bigl(\theta^{-1}\log(1+\lambda \theta  b)\bigr)\bigr)\Big),
\end{eqnarray} 
and $\theta \alpha(\lambda):= \log(1+\lambda \theta b) - \log(1+\lambda \theta c)$ \footnote{We should remark that, following our convention, Vogel uses a commutative Rota--Baxter algebra of weight $\theta = -1$.}. It is obvious that Vogel's identity simply translates into the dendriform identity (\ref{eq:degree1relation}): 
$$
	X=a+\lambda X \succ b - \lambda c \prec X,
$$
with its general solution  $X = Y * \bigl( Y^{-1} \succ a  \prec Z^{-1} \bigr) * Z $, $Y=Y(-c)$ and $Z=Z(b)=Y(-b)^{-1}$, i.e.: 
$$
X =  \exp^*\bigl(\Omega'(- \lambda c)\bigr) * \Big( \exp^*\bigl(-\Omega'(- \lambda c)\bigr) \succ
	a \prec \exp^*\bigl((\Omega'(- \lambda b))\bigr)\Big) * \exp^*\bigl(- \Omega'(- \lambda b)\bigr).
$$
Rewriting it into the Rota--Baxter representation this gives the solution to Vogel's identity in a noncommutative Rota--Baxter algebra. It is an interesting exercise to retrieve Vogel's solution from the dendriform one in the case of a commutative Rota--Baxter algebra. Hence, first observe that in this case:
$$
	a \rhd b = a \succ b - b \prec a = R(a)b - (-b\tilde{R}(a))= R(a)b -  \theta ba - bR(a)=[R(a),b] - ba = -  \theta ba.
$$ 
Such that: 
$$
	\Omega'(\lambda a) = \sum_{n \ge 0} \frac{B_n}{n!} L_\rhd[\Omega']^{(n)}(\lambda a) = \sum_{n \ge 0} \frac{B_n}{n!} (- \theta)^n \lambda (\Omega'(\lambda a))^{n} a =  - \frac{ \log(1-\lambda  \theta a)}{ \theta }.
$$
By recalling our convention (\ref{unite}) and using commutativity as well as  the well known identity for the product (\ref{RBassoprod}), $ \theta a * b= \tilde{R}(a)\tilde{R}(b)-R(a)R(b)$, we find:
\begin{eqnarray*}
X \!\!\!\!\!&=&\!\!\!\! \exp\bigl( \tilde{R}(\frac{\log(1 + \lambda \theta c)}{ \theta })\bigr) \exp\bigl(- \tilde{R}(\frac{\log(1 + \lambda \theta b)}{ \theta })\bigr)
    \tilde{R}\Bigl(a  \exp\bigl(R (\frac{\log(1 + \lambda  \theta c)}{ \theta }) +\tilde{R}(\frac{\log(1 + \lambda \theta  b)}{ \theta })\bigr)\Bigr)\\
   &&\!\!- \exp\bigl( -R(\frac{\log(1 + \lambda \theta  c)}{ \theta })\bigr) \exp\bigl( R(\frac{\log(1 + \lambda \theta  b)}{ \theta })\bigr)
    R\Bigl( a \exp\bigl(R (\frac{\log(1 + \lambda \theta  c)}{ \theta }) +\tilde{R}(\frac{\log(1 + \lambda \theta  b)}{ \theta })\bigr)\Bigr).
 \end{eqnarray*}


\subsection{Rooted trees and the coefficients in the pre-Lie Magnus expansion}
\label{ssect:trees}

The final section is devoted to a closer analysis of the pre-Lie Magnus expansion. Iserles and N{\o}rsett \cite{IsNo99} used planar rooted binary trees to disentangle the combinatorial structure underlying the classical Magnus expansion, see also \cite{Iserles00,Iserles}. Here we propose the use of planar rooted trees to represent the pre-Lie Magnus expansion. This, of course, incorporates the classical case but, is more in the line of --non-commutative-- Butcher series. 

However, let us first use planar rooted binary trees. Recall $\Omega'(\lambda a ) =\sum_{n>0} \lambda^n \Omega'_n(a)$:
$$
	\Omega'(\lambda a) = \sum_{n\ge 0} \frac{B_n}{n!} L^{(m)}_\rhd[\Omega'](\lambda a).
$$ 
This writes out as:
\begin{align*}
	\Omega'(\lambda a) &= \lambda a - \lambda^2 \frac 12 a \rhd a 
							+ \lambda^3 \frac14 (a \rhd a) \rhd a + \lambda^3 \frac{1}{12}  a \rhd (a \rhd a) - \lambda^4 \frac 18  ((a \rhd a) \rhd a ) \rhd a \\
	& -\lambda^4 \frac{1}{24}\Big( (a \rhd (a \rhd a )) \rhd a + a \rhd ((a \rhd a ) \rhd a) +  (a \rhd a) \rhd (a  \rhd a)  \Big) +\ldots
\end{align*}
At fifth order we have ten terms. Using planar binary rooted trees to encode the pre-Lie product: 
$$
    \tree \quad \sim \ a \rhd a,
$$
we find at fourth order:
$$
    -\frac{1}{8}\ \treeC  \hskip 8mm
    -  \hskip 3mm \frac{1}{24}\biggl(
    \hskip 1mm
    \treeG  \hskip 8mm +  \hskip 3mm
    \treeF  \hskip 8mm +  \hskip 3mm
    \treeD  \hskip 7mm
    \bigg).
$$
In \cite{EM} we observed that from order four upwards the left pre-Lie relation (\ref{prelie1}) implies a reduction in the number of terms. Indeed, at order four we observe a reduction from the four terms above to the following two terms:
$$
    -\frac{1}{6}\ \treeC \hskip 10mm -  \hskip 3mm \frac{1}{12} \, \treeF
$$
thanks to the left pre-Lie relation:
$$
    \treeD \hskip 8mm - \hskip 2mm \treeC
           \hskip 8mm = \hskip 2mm
    \treeF \hskip 8mm - \hskip 2mm \treeG 
$$ 

To go beyond this order we propose another approach in the spirit of non-commutative Butcher series.
In fact, one can write the pre-Lie Magnus expansion like a classical Butcher series using ordinary planar rooted trees. By definition a planar rooted tree $t$ is made out of vertices and nonintersecting oriented edges such that all but one vertex have exactly one incoming line. For a tree $t$ we denote by $V(t)$ and $E(t)$ the sets of vertices and edges, respectively. By $f(v)$, $v \in V(t)$ we denote the number of outgoing edges, i.e. the fertility of the vertex $v$ of the rooted tree $t$. The degree of a tree $deg(t)$ is given by its number of vertices. We denote the graded vector space of planar rooted trees by $\mathcal{T}_{pl}:=\oplus_{m > 0}\mathcal{T}^{(m)}_{pl}$. Let $t=B_+(t_1,\cdots, t_n)$, where $B_+$ denotes the grafting operation, joining the $n$ roots via $n$ new edges to a new vertex, which becomes the root of the new tree $t$ and $deg(t)=\sum_{i=1}^n deg(t_i) + 1$. Now, we define the linear map $\alpha: \mathcal{T}_{pl} \to k$:
$$
	\alpha(t): = \frac{B_n}{n!} \prod_{i=1}^{n} \alpha(t_i) =  \prod_{v \in V(t)}  \frac{B_{f(v)}}{f(v)!}.
$$ 
We see immediately that this function maps a rooted tree $t$ containing vertices of fertility $f(v)=2n+1$, $n>0$, $v \in V(t)$ to zero.

\begin{lem} 
\label{lem:kernelalpha}
	$ker(\alpha)=\{t \in \mathcal{T}_{pl} \, |\;  \exists v\in V(t) : f(v)=2n+1, n>0\}$
\end{lem}

In the following we denote by $\mathcal{T}^{e1}_{pl}$ the graded vector subspace of planar rooted trees excluding trees $t$ with vertices of fertility $f(v)=2n+1$, $n>0$, $v\in V(t)$. Recall 
$
   r_{\rhd}^{(n)}(a_1,\dots,a_n) :=
    a_1 \rhd \bigl( a_2 \rhd \bigl( a_3 \rhd
    \cdots (a_{n-1} \rhd a_n) \bigr) \cdots \bigr)$.
Next, we introduce the tree functional $F$,  $F[\ta1\ ](a)=a$ and for  $t=B_+(t_1\cdots t_n)$, $n \ge 1$:
$$
	F[t](a):= r_{\rhd}^{(n+1)}\big(F[t_1](a),\dots,F[t_n](a),a\big). 
$$
\begin{thm}
The pre-Lie Magnus expansion can be written:
\begin{equation}\label{plmplanar}
	\Omega'(a) = \sum_{t \in \mathcal{T}^{e1}_{pl}} \alpha(t) F[t](a).
\end{equation}
\end{thm}
\begin{proof}
The fact that the sum only includes planar rooted trees with even vertices and vertices of fertility one follows form Lemma \ref{lem:kernelalpha}. We just have to prove that the RHS of \eqref{plmplanar} verifies equation \eqref{main3} (with $\lambda$ set to $1$). Indeed, if we denote by $\Omega''$ this right-hand side we have:
\begin{eqnarray*}
\Omega''&=&\sum_{t \in \mathcal{T}^{e1}_{pl}} \alpha(t) F[t](a)\\
	&=&
	\sum_{m\ge 0}\ \ \sum_{t_1,\ldots t_m\in
  	\mathcal{T}^{e1}_{pl}}\frac{B_m}{m!}\alpha(t_1)\cdots\alpha(t_m)r_\rhd^{(m+1)}\big(F[t_1](a),\dots,F[t_n](a),a\big)\\
	&=&
	\sum_{m\ge  0}
	\frac{B_m}{m!}r_\rhd^{(m+1)} 
	\big(\mathop{\underbrace{\Omega'',\ldots,\Omega'' }}
	\limits_{m \hbox{ \eightrm times}}, a\big),
\end{eqnarray*}
hence $\Omega''=\Omega'$.
\end{proof}
\begin{rmk}
Let $T(z)=\sum_{k\ge 0}T_kz^k$ be the Poincar\'e series of
$\mathcal{T}^{e1}_{pl}$. Using the methods of \cite{D} we can prove that $T$
is the solution of the following functional equation:
\begin{equation}
T(z)-z^2T(z)^3=\frac{1}{1-z},
\end{equation}
which in turn yields the recursive definition of the coefficients $T_n$:
\begin{equation}
T_n=1+\sum_{p,q,r\ge 0,\,p+q+r=n-2}T_pT_qT_r.
\end{equation}
This sequence matches with
sequence [A049130] in Sloane's encyclopedy \cite{S}. On can see this by noticing that the series $W=zT$ satisfies the functional equation\footnote{We greatly thank Lo\"\i c Foissy for having indicated this point to us}:
\begin{equation}
W(z)-W(z)^3=\frac{z}{1-z},
\end{equation}
hence $W$ is the reversion of the series $z\mapsto\displaystyle \frac{z(z-1)(z+1)}{z^3-z+1}$. The first terms of the sequence are 
$$1,1,2,4,10,26,73,211,630,1918,...$$.
\end{rmk}
This simple formula provides a simple and efficient way to expand Magnus' --pre-Lie-- recursion. Let us check a few examples:
\begin{align*}
	\Omega'(a)&=  \alpha(\ta1\ ) F[\ta1\ ](a) + \alpha(\tb2\ ) F[\tb2\ ](a) 
										+ \alpha(\tc3\ ) F[\tc3\ ](a)
										+ \alpha(\td31\ ) F[\td31\ ](a) +\cdots\\
			    &= a + B_1  r_{\rhd}^{(2)}\big(F[\ta1\ ](a),a\big)
				      + B_1^2 r_{\rhd}^{(2)}\big(F[\tb2\ ](a),a\big)
				      + \frac {B_2 }{2!} r_{\rhd}^{(2)}\big(F[\ta1\ ](a),F[\ta1\ ](a),a\big) + \cdots
\end{align*}
Recall that here $\mathcal{T}_{pl}$ denotes the vector space of planar rooted trees. Hence, at order four 
we find:
\begin{align*}
	\Omega'(a)&= \cdots + \alpha(\te4\ ) F[\te4\ ](a) + \alpha(\tf41\ ) F[\tf41\ ](a) 
										  + \alpha(\tg42\ ) F[\tg42\ ](a)
									          +\alpha(\thj44\ ) F[\thj44\ ](a) + \cdots\\
			    &= \cdots + B^3_1  r_{\rhd}^{(2)}\big(F[\tc3\ ](a),a\big)
				           + \frac {B_2 }{2!}B_1 r_{\rhd}^{(2)}\big(F[\tb2\ ](a),F[\ta1\ ](a),a\big)
				           + \frac {B_2 }{2!}B_1 r_{\rhd}^{(2)}\big(F[\ta1\ ](a),F[\tb2\ ](a),a\big)\\ 
				 &  \hspace{1cm} \;\;\;           
				           + B_1\frac {B_2 }{2!} r_{\rhd}^{(2)}\big(F[\td31\ ](a),a\big)
+ \cdots.
\end{align*}
Recall that $B_3=0$. Due to the recursive nature of the pre-Lie Magnus expansion at order five the following set of $8$ trees appear:
$$
	\bigl\{ B_+(\te4),\ B_+(\tg42),\ B_+(\ta1,\ta1,\ta1,\ta1) ,\ B_+(\td31,\ta1),\ B_+(\ta1,\td31),\ B_+(\tb2,\tb2),\ B_+(\tc3,\ta1),\ B_+(\ta1,\tc3) \bigr\}.
$$ 
Of course, the first two terms have different coefficients. Other trees do not appear due to the fact that odd Bernoulli numbers are zero, i.e. $B_3=0$. Hence, for the order five contribution we find:
\begin{align*}
	\Omega'_5(a) &=  
	 - B_1 \frac 16\  \bigl( ((a \rhd a) \rhd a ) \rhd a \bigr) \rhd a  
	 - B_1 \frac{1}{12}\ \bigl( a \rhd ((a \rhd a ) \rhd a) \bigr) \rhd a \\
	& + \frac{ B_4}{4!}\ a \rhd \bigl(a \rhd (a \rhd (a \rhd a)) \bigr) 
	    + \frac{B^2_2}{2!2!}\ (a  \rhd (a \rhd a)) \rhd (a \rhd a)  + \\
	 & + \frac{B^2_2}{2!2!}\ a \rhd  \bigl((a  \rhd (a \rhd a)) \rhd a \bigr)  
	    + \frac{B_2}{2!} B_1^2\ (a \rhd a) \rhd ((a  \rhd a)  \rhd a )  \\
	& + \frac{B_2}{2!} B_1^2\ ((a\rhd a) \rhd a) \rhd (a \rhd a ) 
	   + \frac{B_2}{2!} B_1^2\ a \rhd \bigl( ((a\rhd a) \rhd a) \rhd a \bigr) .
\end{align*}
Using the left pre-Lie relation we may write: 
\begin{align*}
            &  \frac{B_2}{2!} B_1^2\ ((a\rhd a) \rhd a) \rhd (a \rhd a ) \\
	   = &\frac{B_2}{2!} B_1^2\ a \rhd \bigl( ((a\rhd a) \rhd a) \rhd a \bigr)
             - \frac{B_2}{2!} B_1^2\ \bigl( a \rhd ((a\rhd a) \rhd a) \bigr) \rhd a
            +\frac{B_2}{2!} B_1^2\ \bigl( ((a\rhd a) \rhd a) \rhd a \bigr) \rhd a   
\end{align*}
leading to seven terms:
\begin{align*}
	\Omega'_5(a) &=  
	 - B_1\frac{B_2}{2!} \frac 52\  \bigl( ((a \rhd a) \rhd a ) \rhd a \bigr) \rhd a  
	 - B_1 \frac{B_2}{2!} \frac 12 \ \bigl( a \rhd ((a \rhd a ) \rhd a) \bigr) \rhd a \\
	& + \frac{ B_4}{4!}\ a \rhd \bigl(a \rhd (a \rhd (a \rhd a)) \bigr) 
	    + \frac{B^2_2}{2!2!}\ (a  \rhd (a \rhd a)) \rhd (a \rhd a)  + \\
	 & + \frac{B^2_2}{2!2!}\ a \rhd  \bigl((a  \rhd (a \rhd a)) \rhd a \bigr)  
	    + \frac{B_2}{2!} B_1^2\ (a \rhd a) \rhd ((a  \rhd a)  \rhd a )  \\
	& + B_2 B_1^2\ a \rhd \bigl( ((a\rhd a) \rhd a) \rhd a \bigr) .
\end{align*}
It is an open question whether one can further reduce the order five term using the left pre-Lie identity. Moreover, it is it would be interesting to find a recursion for $\Omega'(\lambda a)$ which already incorporates the pre-Lie identity.

\vspace{0.5cm}
\goodbreak

\textbf{Acknowledgements}
We would like to thank A. Iserles, A. Lundervold, H. Munthe--Kaas and A. Munthe-Kaas-Zanna for helpful discussions and useful remarks.

\end{document}